\documentclass{llncs}
\usepackage{amsmath}
\usepackage{amssymb}
\usepackage{enumerate}
\usepackage{varioref}
\usepackage{charter,eulervm}
\usepackage{graphicx}
\usepackage{boxedminipage}
\usepackage{subfigure}

\newcommand{\es}{\varnothing}

\title{\sc {The Domination Number of Generalized Petersen Graphs with a Faulty Vertex}}

\author{Jun--Lin Guo\inst{1}
\and
 Kuo--Hua Wu\inst{1}
\and
 Yue-Li~Wang\thanks{National
Science Council of Taiwan Support Grant NSC~100--2221--E--011--067-MY3 and NSC~101--2221--E--011--038--MY3.}%
\inst{1}
 \and
Ton Kloks\thanks{National
Science Council of Taiwan Support Grant
NSC~99--2218--E--007--016.}%
\inst{2}
}
\institute{
Department of Information Management\\
 National Taiwan University of Science and Technology\\
 No.~43, Sec.~4, Keelung Rd., Taipei, 106, Taiwan\\
 {\tt ylwang@cs.ntust.edu.tw}
\and
Department of Computer Science\\
 National Tsing Hua University\\
 No.~101, Sec.~2, Kuang Fu Rd., Hsinchu, Taiwan}

\begin{document}

\maketitle

\begin{abstract}
In this paper, we investigate the domination number of generalized
Petersen graphs $P(n, 2)$ when there is a faulty vertex. Denote by
$\gamma(P(n,2))$ the domination number of $P(n,2)$ and
$\gamma(P_f(n,2))$ the domination number of $P(n,2)$ with a faulty
vertex $u_f$. We show that $\gamma(P_f(n,2))=\gamma(P(n,2))-1$ when
$n=5k+1$ or $5k+2$ and $\gamma(P_f(n,2))=\gamma(P(n,2))$ for the
other cases.

\vskip 0.2in \noindent


\end{abstract}

\section{Introduction}
\label{Introduction}

A graph $G$ is an ordered pair $(V(G),E(G))$ consisting of a set
$V(G)$ of vertices and a set $E(G)$ of edges. When the context is
clear, $V(G)$ and $E(G)$ are simply written as $V$ and $E$,
respectively. The {\em open neighborhood} of vertex $v\in V$ is the
set $N(v) = \{u\in V| uv\in E\}$. The {\em closed neighborhood} of
vertex $v\in V$ is the set $N[v] =N(v)\cup \{v\}$. For a set $S$ of
vertices, $N[S]=\bigcup_{v\in S} N[v]$. A set $S\subseteq V$ is a
{\em dominating set} of $G$ if $N[S]=V$ \cite{Berg62}. The {\em
domination number} of $G$, denoted by $\gamma(G)$, is the
cardinality of a minimum dominating set.

For two natural numbers $n$ and $k$ with $n\geqslant 3$ and
$1\leqslant k \leqslant \left\lfloor \frac{n-1}{2}\right\rfloor$,
the {\em generalized Petersen graph} $P(n,k)$ is a graph on $2n$
vertices with $V(P(n,k))=\{u_i,v_i|1\leqslant i\leqslant n\}$ and
$E(P(n,k))=\{u_iu_{i+1},u_iv_i,v_iv_{i+k}|1\leqslant i\leqslant n\}$
with subscripts modulo $n$ \cite{Bigg93,Coxe50,Watk69}. Hereafter,
all operations on the subscripts of vertices are taken modulo $n$
unless stated otherwise.

In \cite{Behz08}, Behzad, Behzad, and Praeger showed that
$\gamma(P(n,2)) \leqslant \left\lceil \frac{3n}{5} \right\rceil$ for
odd $n\geqslant 3$ and conjectured that $\left\lceil \frac{3n}{5}
\right\rceil$ is exactly the domination number of $P(n,2)$. In
\cite{Ebra09}, Ebrahimi, Jahanbakht, and Mahmoodian (independently,
Yan, Kang, and Xu \cite{Yan09} and Fu, Yang, and Jiang \cite{Fu09})
affirmed that $\gamma(P(n,2))=\left\lceil \frac{3n}{5} \right\rceil$
for $n\geqslant 3$.

In this paper, we are concerned with $\gamma(P_f(n,2))$ when there
is a faulty vertex $u_f$ in $P(n,2)$, where $\gamma(P_f(n,2))$
denotes the domination number of $P(n,2)$ with faulty vertex $u_f$,
i.e., $u_f$ is removed from $P(n,2)$. Thus a faulty vertex cannot be
chosen as a vertex in the dominating set. We shall show that, for
$n\geq 3$,
\[\gamma(P_f(n,2))= \begin{cases}
\gamma(P(n,2))-1 & \quad \text{if $n=5k+1$ or $5k+2$} \\
\gamma(P(n,2)) & \quad \text{otherwise.}
\end{cases}
\]

The {\em alteration domination number} of $G$, denoted by $\mu(G)$,
is the minimum number of points whose removal increases or decreases
the domination number of $G$ \cite{Baue83}. The {\em bondage number}
of $G$, denoted by $b(G)$, is the minimum number of edges whose
removal from $G$ results in a graph with larger domination number
\cite{Fink90}. It can be regarded as the fault tolerance problem
when removing vertices or edges from a graph. Fault tolerance is
also an important issue on engineering \cite{Rand78}. This motivates
us to study the the domination number of $P(n,2)$ with a faulty
vertex. By our result, we can find that the lower and upper bounds
of $\mu(P(n,2))$ are as follows: $\mu(P(n,2))=1$ if $n=5k+1$ or
$5k+2$; otherwise, $\mu(P(n,2))\geq 2$. Moreover, we can find that
$2\leqslant b(P(n,2))\leqslant 3$ if $n=5k$, $5k+3$ or $5k+4$.

This paper is organized as follows. In Section~\ref{Preliminaries}
we review some preliminaries of dominating sets in generalized
Petersen graphs. In Section~\ref{faulty Petersen} some properties
are introduced when there is a faulty vertex in $P(n,2)$.
Section~\ref{main results} contains our main results. We conclude in
Section~\ref{Conclusion}.

\section{Preliminaries}
\label{Preliminaries}

In this paper, we follow the terminology of \cite{Ebra09}. However,
for clarity, we introduce some of them as follows.

Let $P(n,2)-u$ denote the resulting graph after $u$ is removed from
$P(n,2)$. In particular, $P(n,2)-u_f$ is denoted by $P_f(n,2)$ where
$u_f$, for some $1\leqslant f\leqslant n$, is the faulty vertex in
$P(n,2)$. We also use $S+u$ and $S-u$ to denote adding an element
$u$ to a set $S$ and removing an element $u$ from a set $S$,
respectively. In the rest of this paper, $S$ always stands for a
domination set of $P_f(n,2)$. A minimum dominating set of $G$ is
called a {\em $\gamma(G)$-set}. When the graph $G$ is clear from the
context, $\gamma(G)$-set is written as $\gamma$-set. A {\em block}
of $P(n, 2)$ is an induced subgraph of 5 consecutive pairs of
vertices (see Figure \ref{fig1}). Denote by $\mathcal{B}_i$ if a
block of $P(n, 2)$ is centered at $u_i$ and $v_i$. When there is no
possible ambiguity, $\mathcal{B}_i$ and $V(\mathcal{B}_i)$ are used
interchangeably. The vertices of $\mathcal{B}_i-u_i$ can be
partitioned into $R_i=\{v_{i+1},u_{i+2},v_{i+2}\}$,
$L_i=\{v_{i-1},u_{i-2},v_{i-2}\}$, and
$M_i=\{u_{i-1},v_i,u_{i+1}\}$. Let $N^+(R_i)=N[R_i]\setminus
\mathcal{B}_i=\{v_{i+3},u_{i+3},v_{i+4}\}$,
$N^+(L_i)=N[L_i]\setminus
\mathcal{B}_i=\{v_{i-3},u_{i-3},v_{i-4}\}$, and
$\gamma_i(S)=|\mathcal{B}_i\cap S|$. When the context is clear,
$\gamma_i(S)$ is written as $\gamma_i$. Let
$F=\{f-2,f-1,f,f+1,f+2\}$ which contains the indices in
$\mathcal{B}_f$.


\begin{figure}[htb]
\begin{center}
\unitlength=1pt
\begin{picture}(360,120)(0,0)
\put(70,-10){\scriptsize{
\includegraphics[scale=0.5]{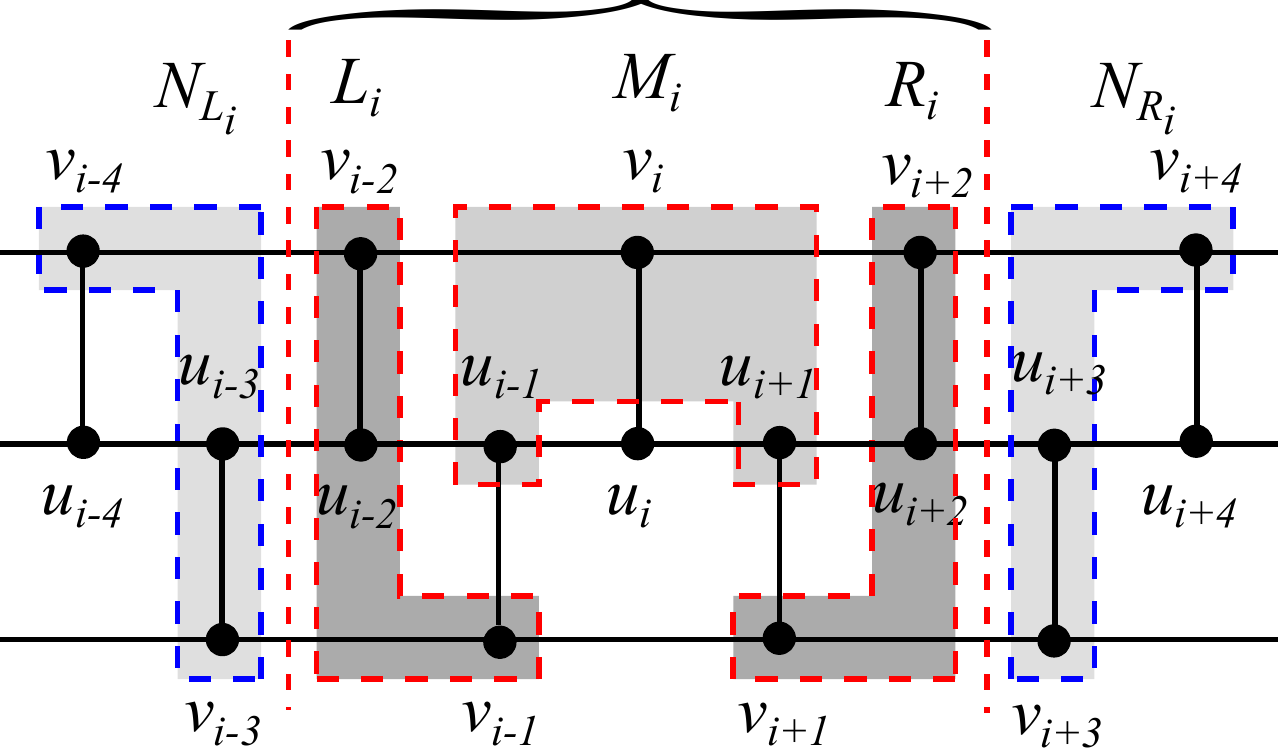}
}}
\put(160,103){$\mathcal{B}_i$}
\end{picture}
\caption{A block $\mathcal{B}_i$.}\label{fig1}
\end{center}
\end{figure}

\newpage

\begin{theorem}[\cite{Ebra09,Fu09,Yan09}]\label{pre:gammaPn2}
For $n \geqslant 3$, $\gamma(P(n, 2))= \left\lceil \frac{3n}{5}
\right\rceil$.
\end{theorem}

\begin{corollary}\label{coro:gammaPn2}
For $n \geqslant 3$, $\left\lceil \frac{3n}{5}
\right\rceil-1\leq\gamma(P_f(n, 2))\leq \left\lceil \frac{3n}{5}
\right\rceil$.
\end{corollary}
\begin{proof} Let $S$ be a $\gamma(P(n,2))$-set. If $u_f\notin S$, then $S$ is also a dominating set of
$P_f(n,2)$ and $\gamma(P_f(n, 2))\leq \left\lceil \frac{3n}{5}
\right\rceil$. For the case where $u_f\in S$, since $\gamma(P_f(n,
2))\leq \left\lceil \frac{3n}{5} \right\rceil$, by symmetry, we can
relabel the subscripts of the vertices in $P(n,2)$ but not $S$ such
that $u_f\notin S$. Thus, in this case, $\gamma(P_f(n, 2))\leq
\left\lceil \frac{3n}{5} \right\rceil$.

To prove that $\left\lceil \frac{3n}{5}
\right\rceil-1\leq\gamma(P_f(n, 2))$, suppose to the contrary that
there exists a dominating set $S$ of $P_f(n, 2)$ with
$|S|=\left\lceil \frac{3n}{5} \right\rceil-2$. It is clear that
$S\cup \{u_f\}$ is also a dominating set of $P(n,2)$ whose
cardinality is $\left\lceil \frac{3n}{5} \right\rceil-1$, a
contradiction. This completes the proof. \qed\end{proof}

\begin{lemma}\label{P1}
 Let $S$ be a minimum dominating set in $P_f(n,2)$ and assume that
 the vertex $u_f$ in the corresponding graph $P(n,2)$ has
 at least one neighbor in $S$. Then $|S|= \left\lceil \frac{3n}{5}
\right\rceil$.
\end{lemma}
\begin{proof}  Since $N(u_f) \cap S\neq \es$, $S$ is also a dominating set of
$P(n,2)$. This implies that $\gamma(P(n, 2))\leqslant |S|$. By
Theorem~\ref{pre:gammaPn2}, there exists a $\gamma(P(n, 2))$-set,
say $T$, with $u_f\notin T$. Clearly, $T$ is also a dominating set
of $P_f(n, 2)$. Thus $|S| \leqslant \gamma(P(n, 2))$. This further
implies that $|S| =\gamma(P(n, 2))$. By Theorem~\ref{pre:gammaPn2},
the lemma follows. \qed\end{proof}

\section{Some properties when there is a faulty vertex}
\label{faulty Petersen}

In this section, we introduce some properties of $\mathcal{B}_f$ in
$P(n,2)$, where $u_f$ is a faulty vertex. By Lemma~\ref{P1}, it
remains to consider the case where $N(u_f) \cap S= \es$. Thus, in
the rest of this paper, we assume that $S$ is a minimum dominating
set under the condition that $N(u_f) \cap S= \es$ unless stated
otherwise. Thus, in this case, $M_f\cap S= \es$ which implies
$\mathcal{B}_f\cap S\subseteq L_f\cup R_f$. Hereafter, when we say
that a vertex $x$ is dominated with respect to $S$, then $x$ is
either in $S$ or $x$ is adjacent to some vertex in $S$.

\begin{proposition}\label{pp:exchange}
Assume that $S$ is a dominating set of graph $G$ and $S'=S-x+y$,
where $x\in S$ and $y\notin S$. If all vertices in $N[x]$ are
dominated by $S'$, then $S'$ is also a dominating set of $G$ with
$|S|=|S'|$.
\end{proposition}

\begin{lemma}\label{lmm:P_i=3}
If there exists $u_i\notin S$ and $u_f\notin M_i$ for some
$1\leqslant i\leqslant n$, then $\gamma_i\geqslant 3$.
\end{lemma}
\begin{proof} Since $N[M_i]=\mathcal{B}_i$, the vertices in $M_i$ can
only be dominated by some vertices in $\mathcal{B}_i$. Note that any
two vertices in $M_i$ have no neighbor in common except $u_i$.
However, $u_i\notin S$ and $u_f\notin M_i$. This results in
$|N[M_i]\cap S|\geqslant 3$. Thus $\gamma_i\geqslant 3$ and the
lemma follows. \qed\end{proof}

\begin{lemma}\label{lmm:P_j>1}
Assume that there exists a minimum dominating set $S$
 such that $N(u_f) \cap S =\es$.
 Then there exists a minimum dominating set $S$ such that
 $N(u_f) \cap S = \es$ and $\gamma_i \geq 2$ for all $1 \leq i \leq n$.
\end{lemma}
\begin{proof} If $\gamma_i > 1$ for
$1\leqslant i\leqslant n$ in $S$, then we are done. Thus we assume
that there exists $\gamma_i =1$ for some $i\neq f$. By
Lemma~\ref{lmm:P_i=3}, $u_i\in S$; otherwise, $\gamma_i \geqslant
2$. We may assume that none of $u_{i+2}$ and $u_{i+3}$ is $u_f$;
otherwise, reverse $\mathcal{B}_i$ so that $L_i$ and $R_i$ are
interchanged. This further implies that all vertices in $R_i$ must
be dominated by some vertices in $N^+(R_i)$. Since any two vertices
in $R_i$ have no common neighbor in $N^+(R_i)$, all vertices in
$N^+(R_i)$ must be in $S$ so that the vertices in $R_i$ can be
dominated. By Proposition~\ref{pp:exchange}, the set
$S'=S-u_{i+3}+u_{i+2}$ is also a minimum dominating set under the
condition that $N(u_f) \cap S= \es$ since the vertices in
$N[u_{i+3}]$ are still dominated by the vertices in $S'$. Note that
$\gamma_j(S')=\gamma_j(S)$ for $1\leqslant j\leqslant n$ except
$j\in\{i,i+5\}$. It is easy to verify that $\gamma_i(S')=2$ and
$\gamma_{i+5}(S')\geqslant 2$. Moreover, $S'$ has one less elements
in $\{j|\gamma_j(S')=1, 1\leqslant j\leqslant n\}$ than that of $S$.
By applying the above process repeatedly until the set
$\{j|\gamma_j(S')=1, 1\leqslant j\leqslant n\}$ becomes empty, this
results in a minimum dominating set with $\gamma_i \geqslant 2$ for
all $1\leqslant i\leqslant n$. This completes the proof.
\qed\end{proof}

\begin{definition}
\label{Type I set} A minimum dominating set $S$ is called a Type I
set if $\gamma_i \geqslant 2$ for $1\leqslant i\leqslant n$. For a
Type I set $S$, the cardinality of the set $\{i|\gamma_i(S)=2,
u_i\notin \mathcal{B}_f\}$ is called its couple number.
\end{definition}

\begin{proposition}\label{pp:L_iandR_i}
Assume that $S$ is a Type I set. If $\gamma_f = 3$ and $N(u_f) \cap
S= \es$, then either $|L_f\cap S|=1$ or $|R_f\cap S|=1$.
\end{proposition}
\begin{proof} Since $\gamma_f = 3$ and $N(u_f) \cap
S= \es$, $\gamma_f=|L_f\cap S|+|R_f\cap S|=3$. To show that either
$|L_f\cap S|=1$ or $|R_f\cap S|=1$, it is equivalent to showing that
$|L_f\cap S|=0$ or $|R_f\cap S|=0$ is impossible. Suppose to the
contrary that $L_f\cap S= \es$ (or $R_f\cap S= \es$). It can be
found that vertex $u_{f-1}$ (or $u_{f+1}$) is not dominated, a
contradiction. \qed\end{proof}

By Proposition~\ref{pp:L_iandR_i}, in the rest of this paper, we
only consider the case where $|L_f\cap S|=1$ and $|R_f\cap S|=2$
when $S$ is a Type I set with $\gamma_f=3$. For the case where
$|L_f\cap S|=2$ and $|R_f\cap S|=1$, we can reverse the generalized
Petersen graph so that it yields $|L_f\cap S|=1$ and $|R_f\cap
S|=2$.

In total, there are nine possible cases for the vertices in
$\mathcal{B}_f\cap S$ when $N(u_f) \cap S= \es$, $\gamma_f=3$, and
$|L_f\cap S|=1$. However, only four of them are feasible (see
Figures~\ref{figB}(a)-(d)). For example, if $|L_f\cap S|=1$ and
$v_{f-2}\in S$, then $u_{f-1}$ is not dominated and it is an
infeasible case.

\begin{figure}[htb]
\begin{center}
\subfigure[]{
\includegraphics[scale=0.35]{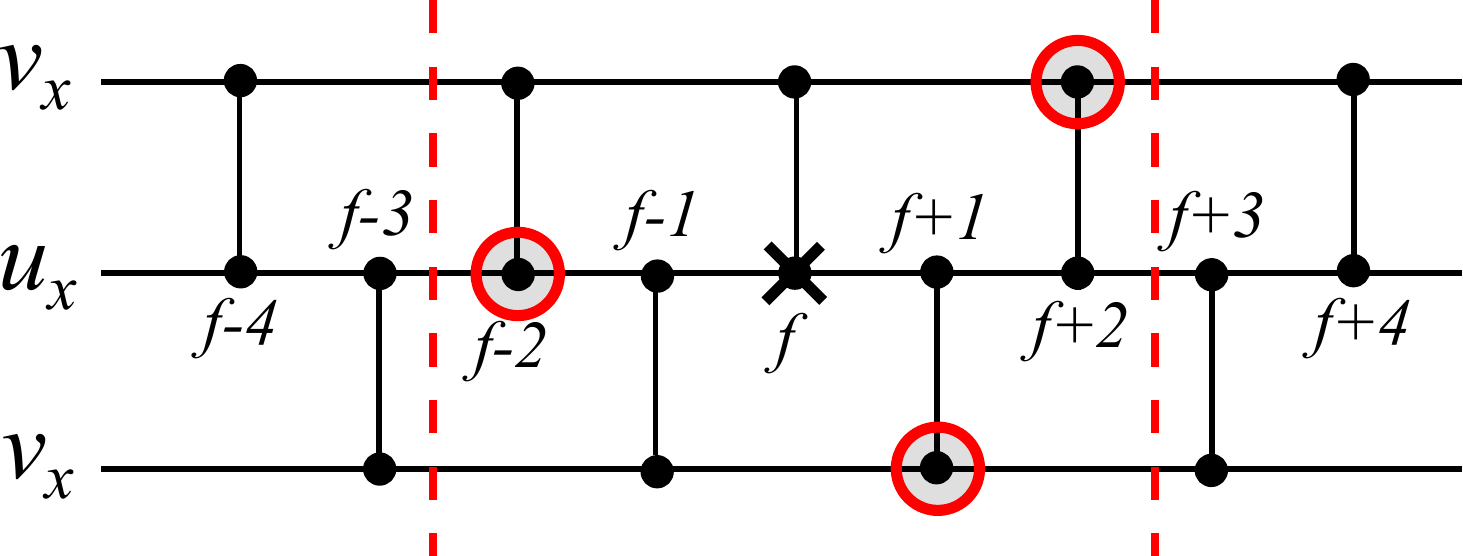}
} \quad \subfigure[]{
\includegraphics[scale=0.35]{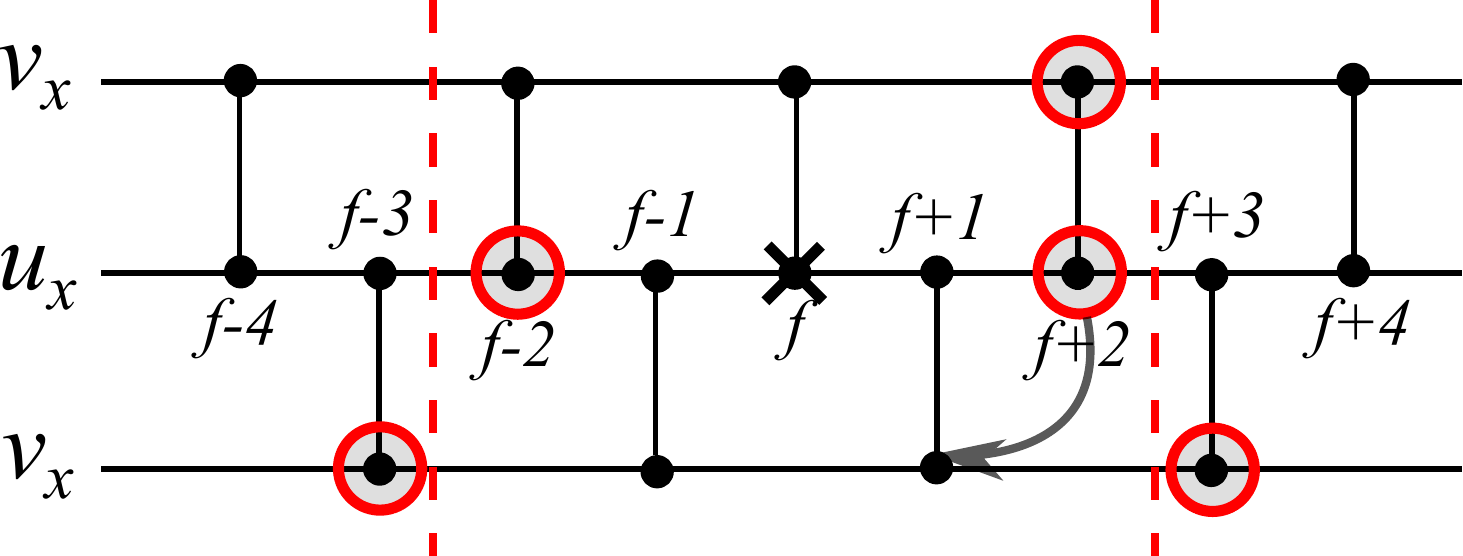}
} \quad \subfigure[]{
\includegraphics[scale=0.35]{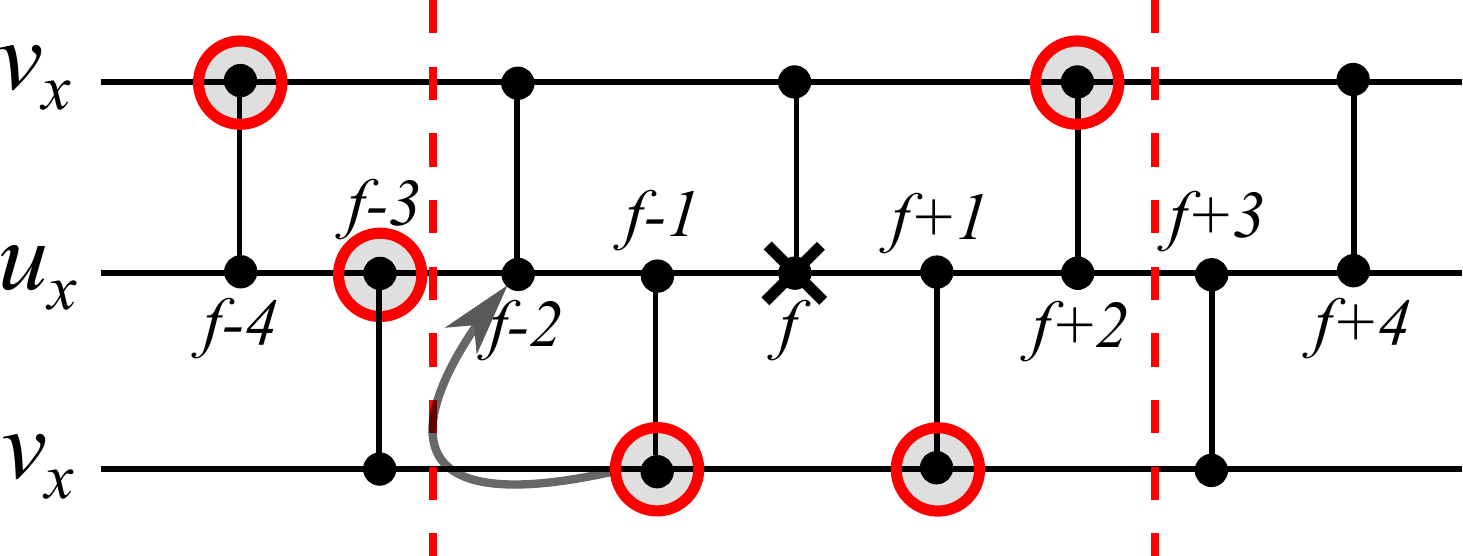}
} \quad \subfigure[]{
\includegraphics[scale=0.35]{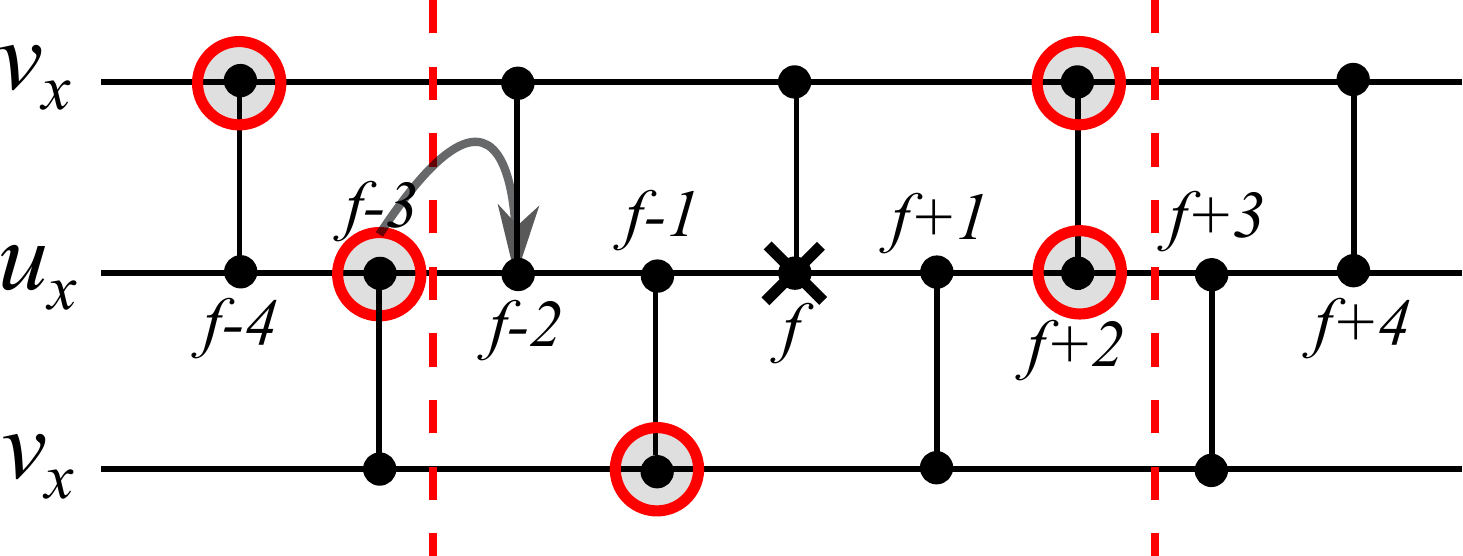}
}  \caption{\label{figB}All feasible $\mathcal{B}_f\cap S$ when $N(u_f)
\cap S= \es$, $\gamma_f=3$, and $|L_f\cap S|=1$.}
\end{center}
\end{figure}

For the case in Figure~\ref{figB}(b), by
Proposition~\ref{pp:exchange}, $S_1=S-u_{f+2}+v_{f+1}$ is still a
Type I set. Note that the pattern of $S_1\cap \mathcal{B}_f$ is
exactly the case in Figure~\ref{figB}(a). For the case in
Figure~\ref{figB}(c), by Proposition~\ref{pp:exchange},
$S_2=S-v_{f-1}+u_{f-2}$ is also a Type I set. Furthermore, the
pattern of $S_2\cap \mathcal{B}_f$ is also exactly the case in
Figure~\ref{figB}(a). For the case in Figure~\ref{figB}(d), by
Proposition~\ref{pp:exchange}, $S_3=S-u_{f-3}+u_{f-2}$ is also a
Type I set while $\gamma_f=4$. We shall define a Type III set later
for this case. Henceforth, if $S$ is a Type I set with $N(u_f) \cap
S= \es$ and $\gamma_f = 3$, then we may assume that it is a Type II
set which is defined as follows.

\begin{definition}
\label{Type II set} A Type I set $S$ with $\gamma_f=3$ is called a
Type II set if $\mathcal{B}_f\cap S=\{u_{f-2},v_{f+1},v_{f+2}\}$
$($see Figure~$\ref{figB}(\mathrm{a}))$.
\end{definition}

Now we consider the case where $N(u_f) \cap S= \es$ and $\gamma_f =
4$. By symmetry, we only need to consider the cases where $|L_f\cap
S|=1$ and $|L_f\cap S|=2$. There are only seven feasible
combinations (see Figures~\ref{fig8}(a)-\ref{fig8}(g)).  Every
minimum dominating set $S$ in the cases of
Figures~\ref{fig8}(e)-\ref{fig8}(g) can be transformed to a Type II
set. That is, set $S_1=S-u_{f+2}+u_{f+3}$ in Figure~\ref{fig8}(e),
set $S_2=S-v_{f-2}+v_{f-4}$ in Figure~\ref{fig8}(f), and set
$S_3=S-v_{f+1}+v_{f+3}$ in Figure~\ref{fig8}(g). Note that
$N[u_{f+2}], N[v_{f-2}]$, and $N[v_{f+1}]$ are still dominated by
the vertices in $S_1, S_2$, and $S_3$, respectively. Thus, by
Proposition~\ref{pp:exchange}, $S_1, S_2$, and $S_3$ are Type II
sets.

\begin{figure}[htb]
\begin{center}
\subfigure[]{
\includegraphics[scale=0.35]{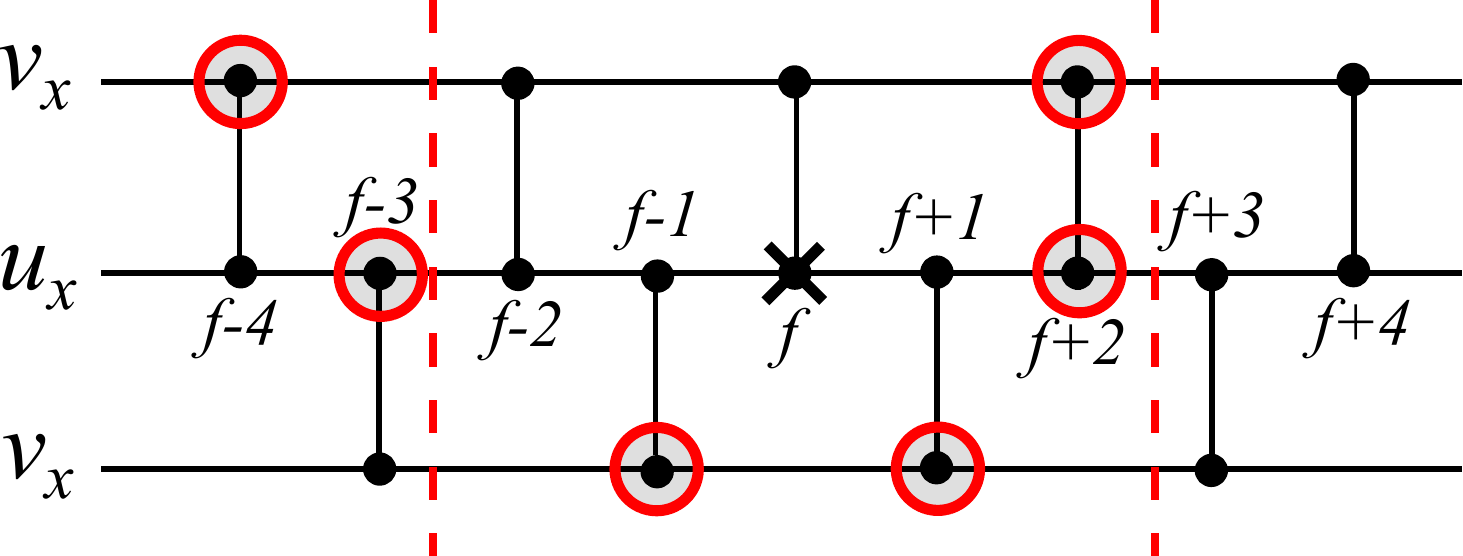}
} \quad \subfigure[]{
\includegraphics[scale=0.35]{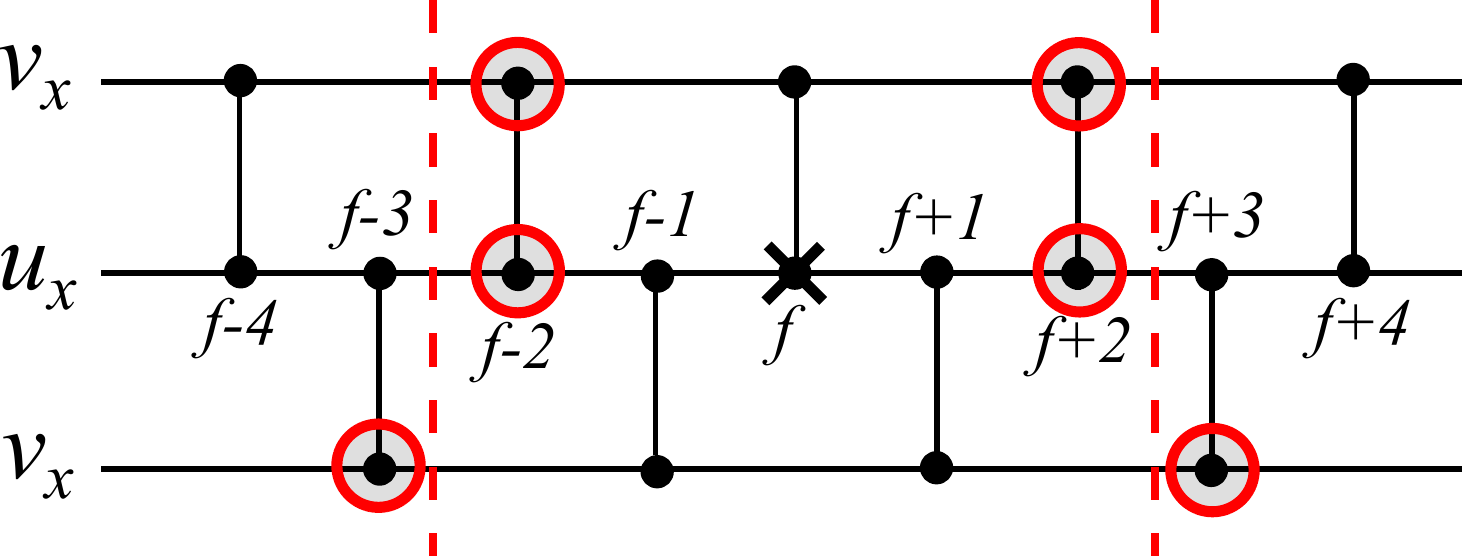}
} \quad \subfigure[]{
\includegraphics[scale=0.35]{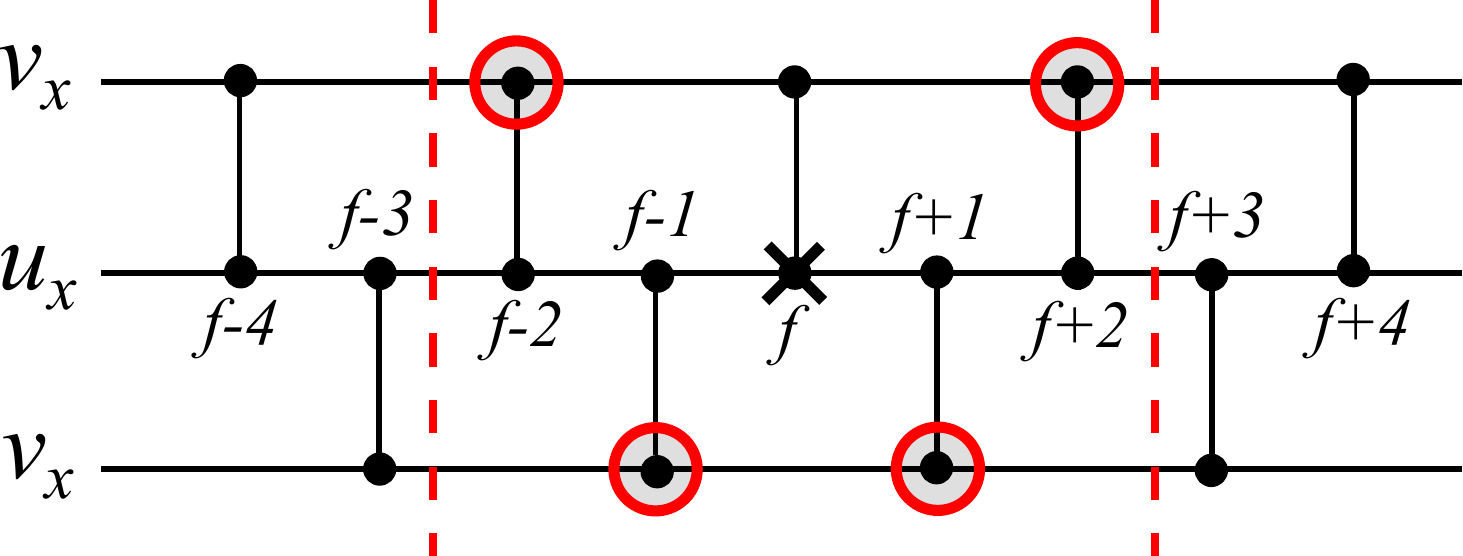}
} \quad \subfigure[]{
\includegraphics[scale=0.35]{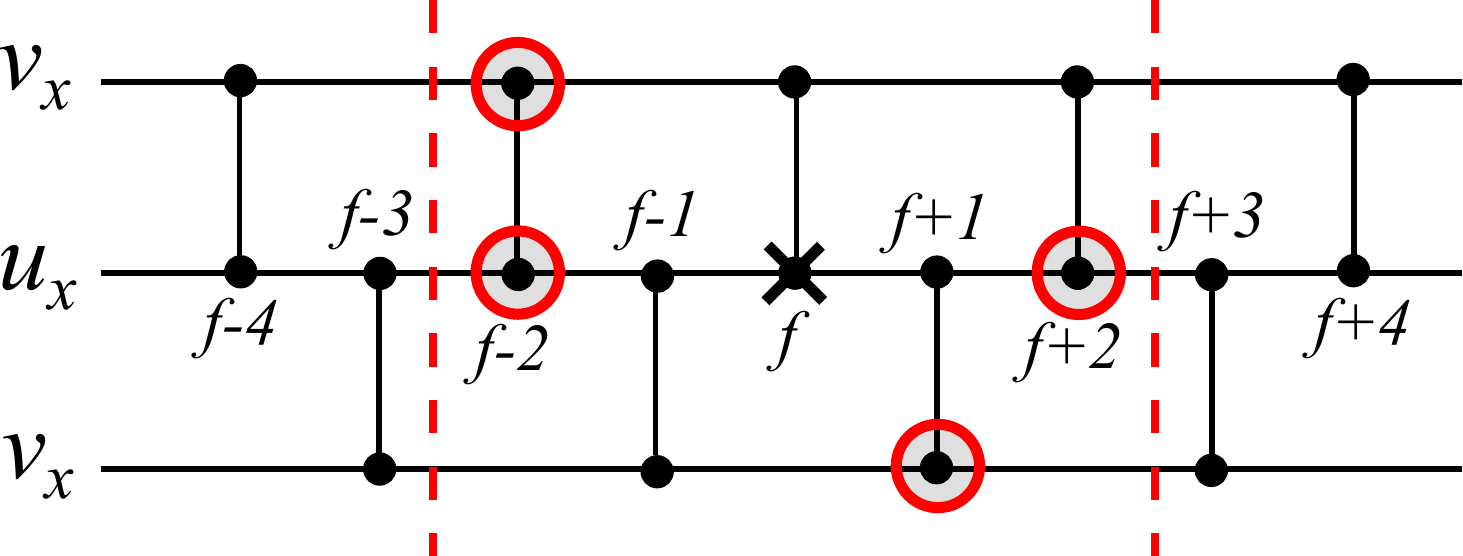}
} \quad \subfigure[]{
\includegraphics[scale=0.35]{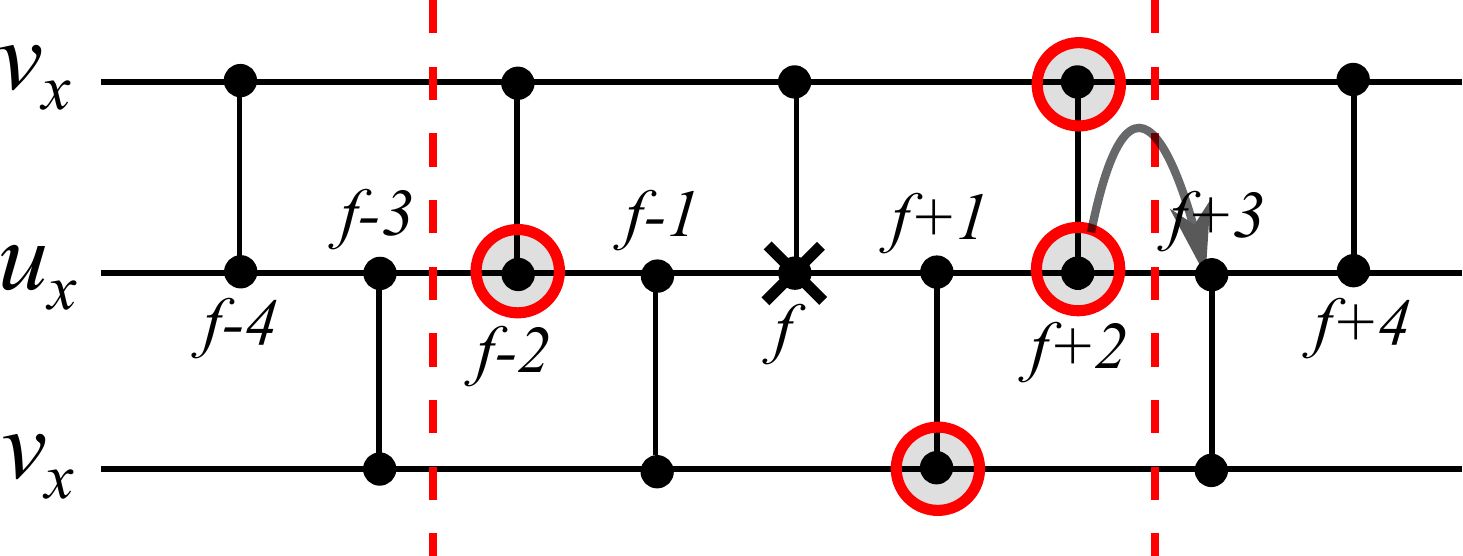}
} \quad \subfigure[]{
\includegraphics[scale=0.35]{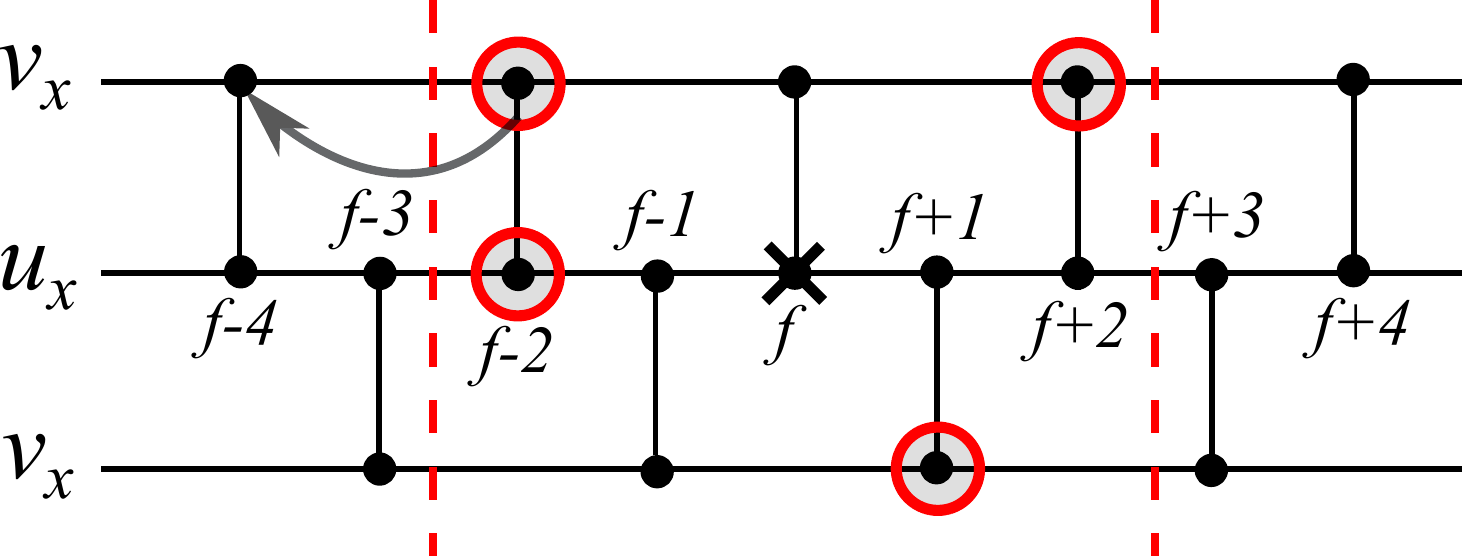}
} \quad \subfigure[]{
\includegraphics[scale=0.35]{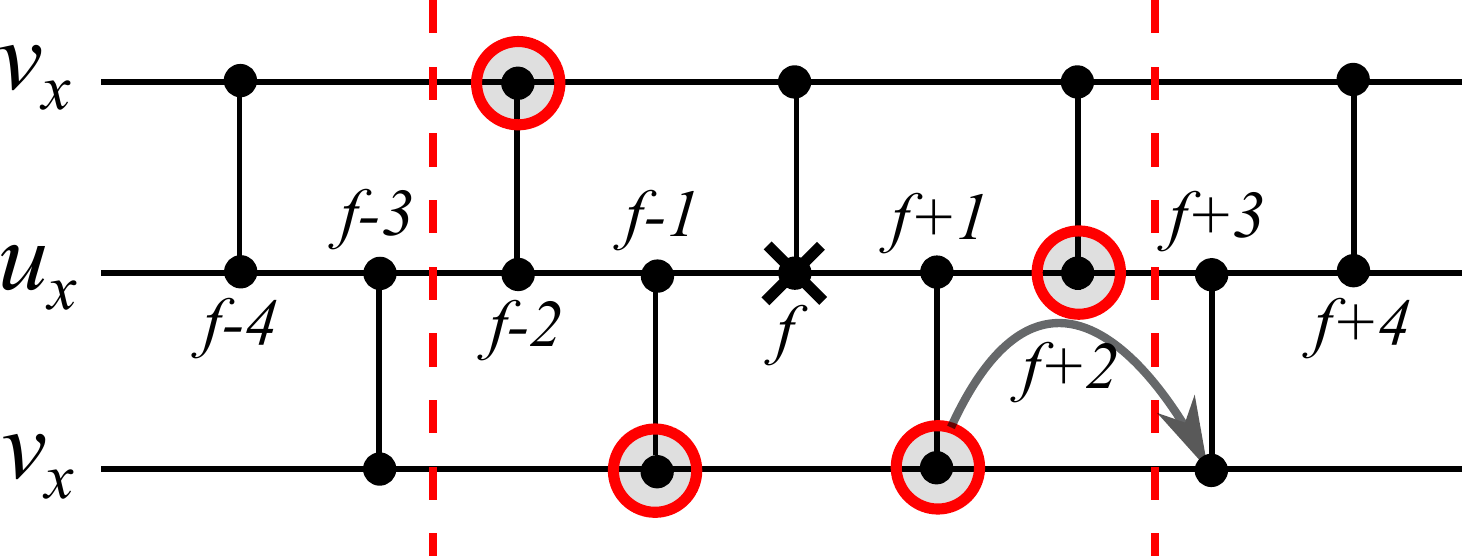}
} \caption{\label{fig8}All feasible $\mathcal{B}_f\cap S$ when $N(u_f)
\cap S= \es$ and $\gamma_f=4$.}
\end{center}
\end{figure}

\begin{definition}
\label{Type II set} A Type I set $S$ with $\gamma_f=4$ is called a
Type III set if $\mathcal{B}_f\cap S$ is equal to one of the
following four sets: $\{v_{f-1},v_{f+1},u_{f+2},v_{f+2}\}$,
$\{u_{f-2},v_{f-2},u_{f+2},v_{f+2}\}$,
$\{v_{f-2},v_{f-1},v_{f+1},v_{f+2}\}$, and
$\{u_{f-2},v_{f-2},v_{f+1},u_{f+2}\}$, $($see
Figures~$\ref{figB}(\mathrm{a})$-$\ref{figB}(\mathrm{d}))$, or
precisely, are called {\em Type III$(\mathrm{a})$-III$(\mathrm{d})$
sets}, respectively.
\end{definition}
Henceforth, if $S$ is a Type I set with $\gamma_f = 4$ and $N(u_f)
\cap S= \es$, then we may assume that it is a Type III set.

\begin{lemma}\label{lmm:gamma=2}
Assume that $S$ is a Type II $($or III$)$ set with couple number
$c$. If there exists $\gamma_i=2$ for $i\notin F$ and $S\cap
\{v_{i-1},v_i,v_{i+1}\}=\es$, then there exists a Type II $($or
III$)$ set with couple number $c-1$.
\end{lemma}
\begin{proof} Clearly, by Lemma~\ref{lmm:P_i=3}, $u_i\in S$. Since
$S\cap\{v_{i-1},v_i,v_{i+1}\}=\es$, the other vertex of
$\mathcal{B}_i$ in $S$, say $y$, must be in
$\{u_{i-2},v_{i-2},u_{i-1}\}\cup \{u_{i+2},v_{i+2},u_{i+1}\}$. We
only consider the case where $y\in \{u_{i-2},v_{i-2},u_{i-1}\}$ (see
Figure \ref{fig2}). The other case is similar. Since $R_i\cap
S=\es$, by using a similar argument as in Lemma~\ref{lmm:P_j>1},
$N^+(R_i)\subset S$. Clearly, at least one vertex in
$\mathcal{B}_{i+5}\setminus N^+(R_i)$ must be in $S$ so that
$u_{i+5}$ and $u_{i+6}$ are dominated. This results in
$\gamma_{i+5}\geqslant 4$ no matter whether $i+5$ or $i+6$ is equal
to $f$ or not. Now let $S'=S-u_{i+3}+u_{i+2}$. It is easy to check
that $\gamma_j(S')=\gamma_j(S)$ for $1\leqslant j\leqslant n$ except
$j\in \{i,i+5\}$. However, $\gamma_i(S')=3$ and
$\gamma_{i+5}(S')\geqslant 3$. Thus the couple number of $S'$ is one
less than that of $S$. This completes the proof. \qed\end{proof}

\begin{figure}[htb]
\begin{center}
\includegraphics[scale=0.35]{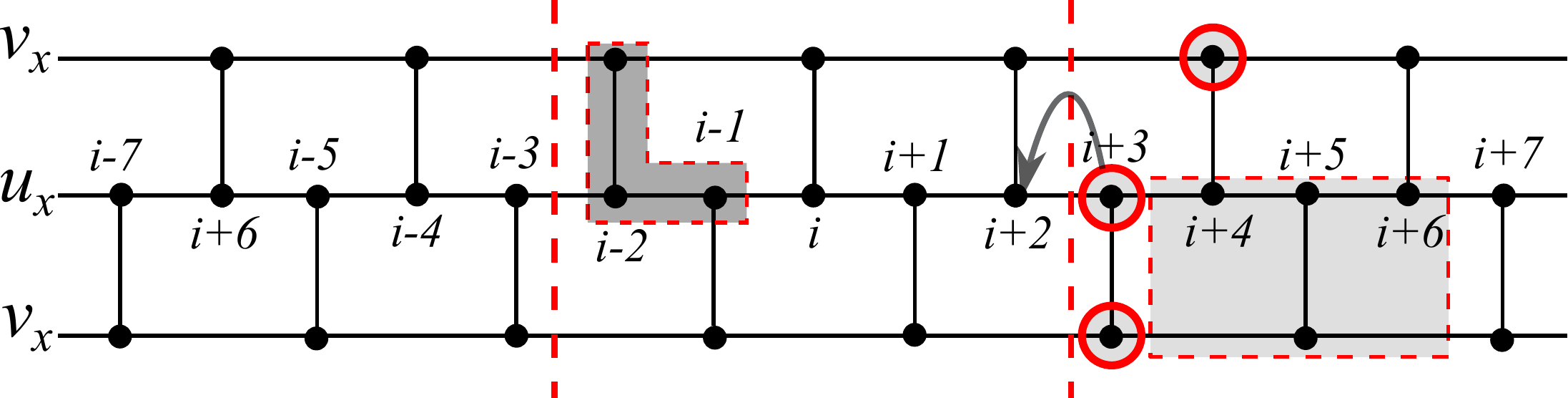}
\caption{\label{fig2}An illustration for Lemma \ref{lmm:gamma=2}.}
\end{center}
\end{figure}

\begin{remark}
Note that if $i=f-2$ $($respectively, $i=f+2)$ in
Lemma~\ref{lmm:gamma=2}, then $u_f\in R_i$ $($respectively, $u_f\in
L_i)$. This yields $u_{i+2}=u_f$ $($respectively, $u_{i-2}=u_f)$
which is removed under our assumption. Thus we cannot obtain a Type
II $($or III$)$ set $S'$ by setting $S'=S-u_{i+3}+u_{i+2}$
$($respectively, $S'=S-u_{i-3}+u_{i-2})$. For the case where $i\in
\{f-1, f+1\}$, $u_i$ even might not be in $S$ since $u_f$ is either
$u_{i+1}$ or $u_{i-1}$ which is removed.
\end{remark}
Hereafter, we assume that $S$ is with the smallest couple number if
$S$ is a Type II or III set.

\begin{definition}
\label{pseudogamma2} Let $S$ be a Type II $($or III$)$ set with the
smallest couple number. A vertex $u_i$ for $i\notin F$ is called a
pseudo-couple vertex with respect to $S$ if $S\cap
\{v_{i-1},v_i,v_{i+1}\}=\es$.
\end{definition}

Notice that if $u_i$ is a pseudo-couple vertex, then $\gamma_i\geq
3$ when $S$ is a Type II $($or III$)$ set with the smallest couple
number.

\begin{lemma}\label{lm:gammageq3}
Assume that $S$ is a Type II or III set with the smallest couple
number. If there exists $\gamma_i=2$ for $i\notin F$, then either
$\gamma_{i+2} \geqslant 4$ or $\gamma_{i-2} \geqslant 4$.
\end{lemma}
\begin{proof}  Since $\gamma_i=2$ and $S$ is with the smallest couple number,
by Lemma \ref{lmm:gamma=2}, $u_i$ cannot be a pseudo-couple vertex
and a vertex $x\in \{v_i,v_{i-1},v_{i+1}\}$ must be in $S$. We
consider the following three cases.

\noindent {\bf Case 1.} $x=v_i$.

It is clear that $v_{i+1}$ and $u_{i+2}$ must be dominated by
$v_{i+3}$ and $u_{i+3}$, respectively. Similarly, $v_{i-1}$ and
$u_{i-2}$ must be dominated by $v_{i-3}$ and $u_{i-3}$, respectively
(see Figure \ref{fig5}(a)). Thus both $\gamma_{i+2}$ and
$\gamma_{i-2}$ are greater than or equal to 4.

\noindent {\bf Case 2.} $x=v_{i+1}$.

In this case,  $u_{i+3}$ and $v_{i+4}$ must be in $S$ so that
$u_{i+2}$ and $v_{i+2}$ are dominated. Thus $\gamma_{i+2}\geqslant
4$ (see Figure \ref{fig5}(b)).

\noindent {\bf Case 3.}  $x=v_{i-1}$.

By using a similar argument as in Case 2, the case holds. This
completes the proof. \qed\end{proof}

\begin{figure}[htb]
\begin{center}
\subfigure[Case 1]{
\includegraphics[scale=0.35]{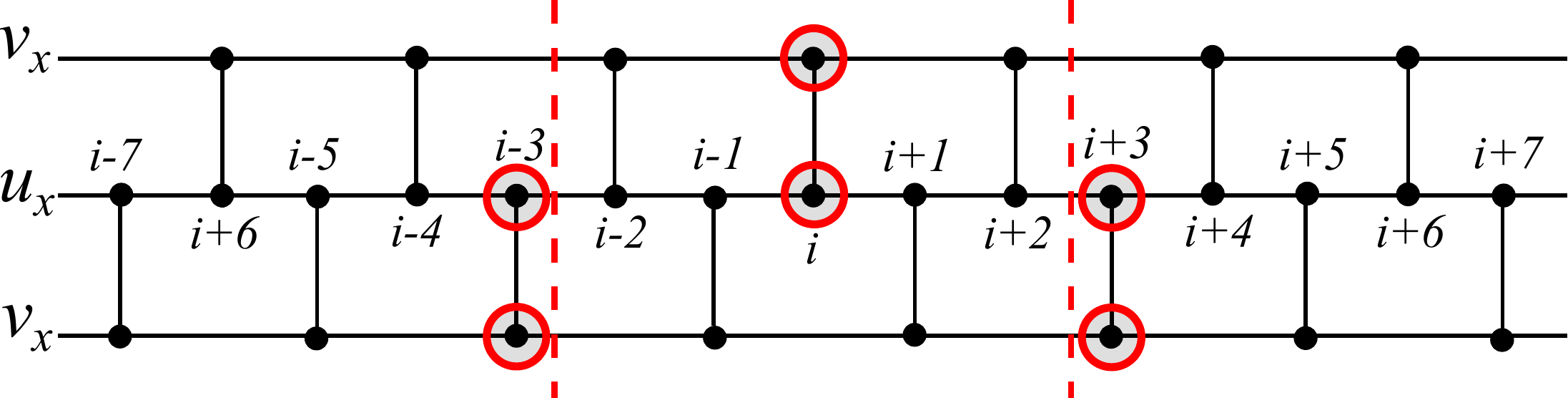}
} \quad \subfigure[Case 2]{
\includegraphics[scale=0.35]{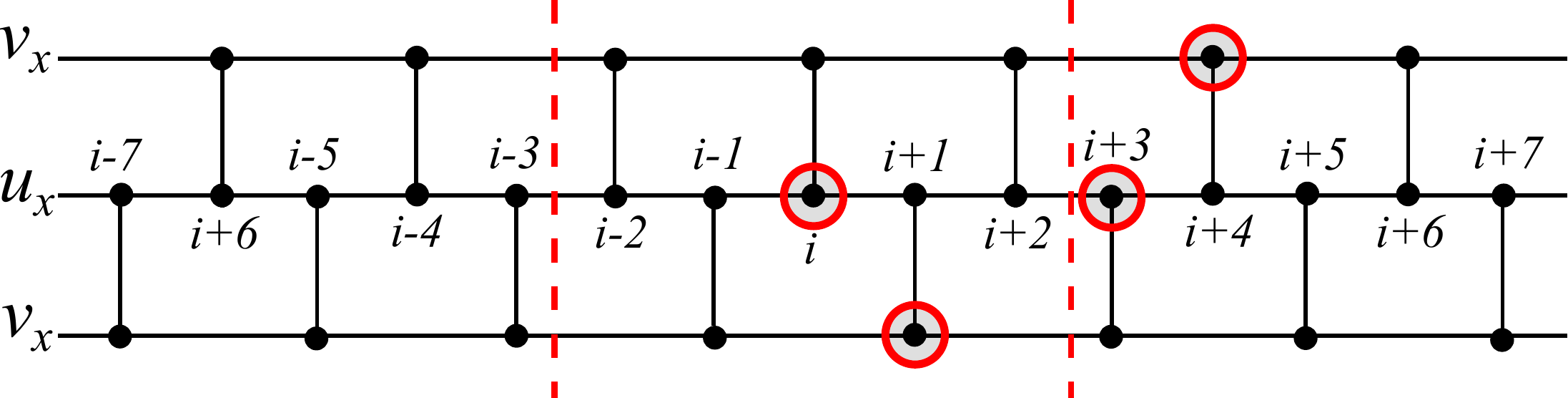}
}  \caption{\label{fig5}Illustrations for Lemma~\ref{lm:gammageq3}.}
\end{center}
\end{figure}

\begin{lemma}\label{lm:j-kneq4}
Assume that $S$ is a Type II or III set with the smallest couple
number. If there exist $\gamma_i=\gamma_j=2$ for distinct $i,j\notin
F$, then $|i-j|\neq 4$.
\end{lemma}
\begin{proof} Suppose to the contrary that $|i-j|=4$. That is, $j$ is either
equal to $i+4$ or $i-4$. We only consider the case where $j=i+4$.
The other case is similar. By Lemma~\ref{lmm:gamma=2}, $u_i$ cannot
be a pseudo-couple vertex and a vertex $x\in
\{v_i,v_{i-1},v_{i+1}\}$ must be in $S$. Analogous to
Lemma~\ref{lm:gammageq3}, we consider the following three cases.

\noindent {\bf Case 1.} $x=v_i$.

We can find that $u_{i+3},v_{i+3}\in \mathcal{B}_{i+4}\cap S$ (see
Figure \ref{fig5}(a)). If $u_{i+4}$ is also in $S$, then
$\gamma_j=\gamma_{i+4}\geqslant 3$, a contradiction. If $u_{i+4}$ is
not in $S$, then, by Lemma~\ref{lmm:P_i=3}, $\gamma_j\geqslant 3$, a
contradiction too.

\noindent {\bf Case 2.} $x=v_{i+1}$.

In this case, $u_{i+3},v_{i+4}\in \mathcal{B}_{i+4}\cap S$ (see
Figure \ref{fig5}(b)). By using a similar argument as in Case 1, we
can find that $\gamma_j\geqslant 3$. Thus this case is also
impossible.

\noindent {\bf Case 3.}  $x=v_{i-1}$.

In this case, all vertices in $N^+(R_i)$ must be in $S$. This
contradicts that $\gamma_j=2$. This concludes the proof of this
lemma. \qed\end{proof}

\section{Main results}
\label{main results}

By Lemma~\ref{lmm:P_i=3}, $\gamma_f\geq 3$. In the following, we
investigate the cardinalities of minimum dominating sets of
$P_f(n,2)$ under all possible values of $\gamma_f$.

By Lemma~\ref{lm:gammageq3}, if there exists $\gamma_i=2$ for
$i\notin F$, then either $\gamma_{i+2} \geqslant 4$ or $\gamma_{i-2}
\geqslant 4$ must hold. This yields
$(\gamma_i+\gamma_{i+2})/2\geqslant 3$ or
$(\gamma_i+\gamma_{i-2})/2\geqslant 3$. By Lemma~\ref{lm:j-kneq4},
if $\gamma_i=2$, then both $\gamma_{i-4}$ and $\gamma_{i+4}$ are
greater than or equal to 3. This means that no two distinct
$\gamma_i=\gamma_j=2$ use the same $\gamma_k$ to obtain the average
number 3. This ensures that the average number of $\gamma_i$ is
greater than or equal to 3 when $i\notin F$. To prove the lower
bound of $|S|$, our main idea is to count the number of $\gamma_i=2$
for $i\in F$, say $x$, which cannot gain support from any other
$\gamma_j$ so that their average is greater than or equal to 3. This
yields $5|S|=\sum_{i=1}^n\gamma_i\geqslant 3n-x$.

\begin{lemma}\label{lm:TypeIIbcd}
If $S$ is a Type III set with the smallest couple number, then
$|S|\geqslant\left\lceil \frac{3n}{5} \right\rceil$.
\end{lemma}
\begin{proof} First we consider the case where $S$ is a Type III(a), III(b), or
III(c) set. By inspection (see Figure~\ref{fig8}), it can be found
that $\gamma_i\geqslant 3$ for $i\in F$ if $S$ is a Type III(a) or
III(b) set. For the case where $S$ is a Type III(c) set, one of the
elements in $\{v_{f-3},u_{f-3},u_{f-4}\}$ (respectively,
$\{v_{f+3},u_{f+3},u_{f+4}\}$) must be in $S$ so that $u_{f-3}$
(respectively, $u_{f+3}$) is dominated. Thus $\gamma_i\geqslant 3$
for $i\in F$ if $S$ is a Type III(c) set. This also implies that, in
those three types of dominating sets, if there exists $\gamma_i=2$
in $S$, then, by Lemmas~\ref{lm:gammageq3} and \ref{lm:j-kneq4},
either $(\gamma_i+\gamma_{i+2})/2 \geqslant 3$ or
$(\gamma_i+\gamma_{i-2})/2 \geqslant 3$ must hold. As a consequence,
$\sum_{j=1}^n \gamma_j=5|S|\geqslant 3n$. This yields
$|S|\geqslant\left\lceil \frac{3n}{5} \right\rceil$.

To complete the proof, it remains to consider the case where $S$ is
a Type III(d) set. According to the possible values of
$\gamma_{f-2}$, we consider the following two cases.

\noindent {\bf Case 1.} $\gamma_{f-2}=2$.

Since $\gamma_{f-2}=2$, vertices $u_{f-5}$ and $v_{f-5}$ must be in
$S$ so that $u_{f-4}$ and $v_{f-3}$ are dominated (see
Figure~\ref{fig9}(a)). This further implies that both $\gamma_{f-3}$
and $\gamma_{f-4}$ are greater than or equal to 4. Note that if
$\gamma_{f-5}=2$, then $u_{f-8}$ and $v_{f-8}$ must be in $S$ so
that $u_{f-7}$ and $v_{f-6}$ are dominated. Accordingly,
$\gamma_{f-7}\geqslant 4$. Thus $\gamma_{f-5}$ can gain support from
$\gamma_{f-7}$ such that $(\gamma_{f-5}+\gamma_{f-7})/2\geqslant 3$.
We can find that the minimum values of $\gamma_{f-4}, \gamma_{f-3},
\ldots, \gamma_{f+2}$ are $4,4,2,3,4,2$, and $2$, respectively. Thus
every $\gamma_i=2$ in $\mathcal{B}_f$ can gain support from a vertex
$\gamma_j=4$. Thus $\sum_{i=1}^n\gamma_i=5|S|\geqslant 3n$ which
yields $|S|\geqslant \left\lceil \frac{3n}{5} \right\rceil$. Thus
this case holds.

\noindent {\bf Case 2.} $\gamma_{f-2}\geqslant 3$.

If $u_{f-3}\in S$ or $v_{f-3}\in S$, then, after setting
$S'=S-u_{f-2}+v_{f-1}$ and $S''=S'-v_{f+1}+v_{f+3}$, $S''$ becomes a
Type II set which will be considered in Lemma~\ref{lm:ri=3}. We may
assume that $u_{f-3}, v_{f-3} \notin S$ and either $u_{f-4}$ or
$v_{f-4}$ is in $S$ (see Figure \ref{fig9}(b)). Note that, in this
case, $v_{f-5}$ must be in $S$ so that $v_{f-3}$ is dominated. Thus
the minimum values of $\gamma_{f-4}, \gamma_{f-3}, \ldots,
\gamma_{f+2}$ are $4,4,3,3,4,2$, and $2$, respectively. Clearly,
their average is greater than or equal to 3. Note that one of the
vertices in $N[u_{f-6}]$ must be in $S$ so that $u_{f-6}$ is
dominated. Thus both $\gamma_{f-5}$ and $\gamma_{f-6}$ are greater
than or equal to 3. This ensures that they will not gain support
from $\gamma_{f-4}$ and $\gamma_{f-3}$ on computing the average
value 3. Hence $\sum_{i=1}^n\gamma_i=5|S|\geqslant 3n$ and
$|S|\geqslant \left\lceil \frac{3n}{5} \right\rceil$. This completes
the proof. \qed\end{proof}

\begin{corollary}\label{lm:gamma_i>4}
If $S$ is a minimum dominating set with the smallest couple number
under the condition that $N(u_f) \cap S= \es$ and $\gamma_f \geq 4$,
then $|S|\geqslant \left\lceil \frac{3n}{5} \right\rceil$.
\end{corollary}
\begin{proof}
Since the case where $S$ is a Type III set with the smallest couple
number is a special case of $\gamma_f \geq 4$, by
Lemma~\ref{lm:TypeIIbcd}, this corollary follows.  \qed\end{proof}

\begin{figure}[htb]
\begin{center}
\subfigure[$\gamma_{f-2}=2$]{
\includegraphics[scale=0.35]{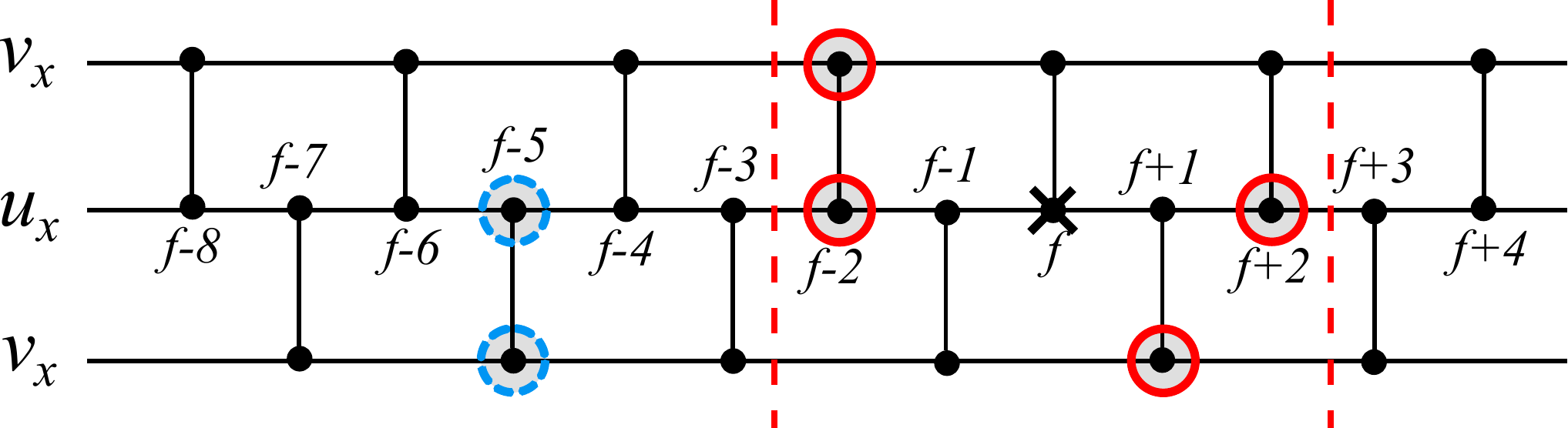}
} \quad \subfigure[$\gamma_{f-2}\geqslant 3$]{
\includegraphics[scale=0.35]{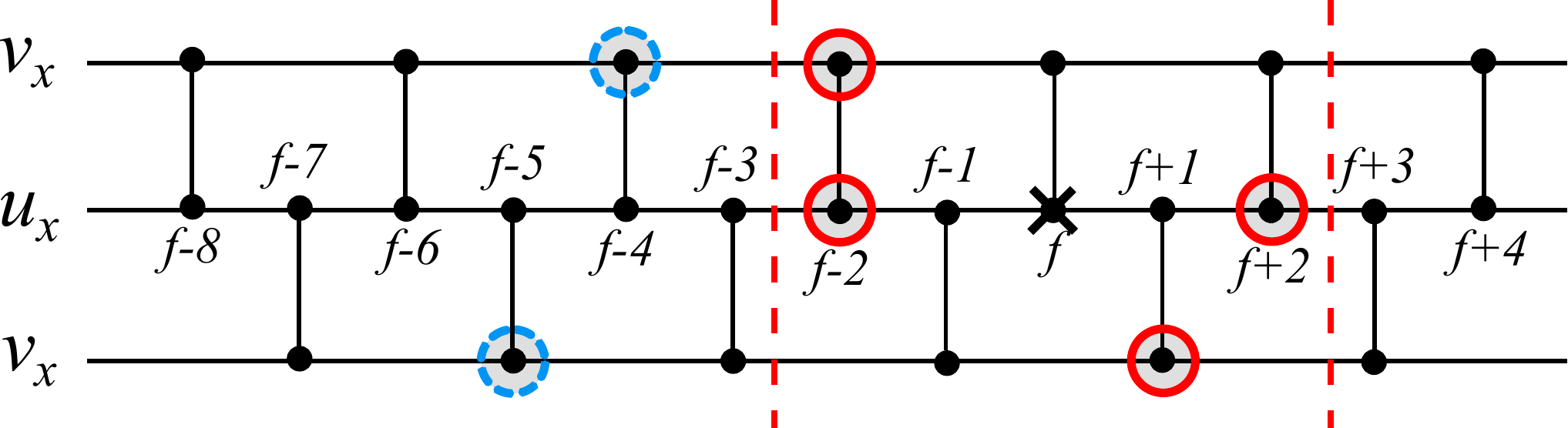}
}  \caption{\label{fig9}Illustrations for Lemma~\ref{lm:gammageq3}.}
\end{center}
\end{figure}

It remains to investigate lower bounds for type II sets. By using a
similar classification in \cite{Ebra09}, we consider the following
five classes: $n=5k, 5k+1, 5k+2, 5k+3$, and $5k+4$.

\begin{lemma}\label{5nn3}
For $n=5k$ or $5k+3$, if $S$ is a Type II set, then
$|S|\geqslant\left\lceil \frac{3n}{5} \right\rceil$.
\end{lemma}
\begin{proof} Let $S$ be a Type II set with the smallest
couple number. By inspection on Figure~\ref{figB}(a), only the
elements in $\{\gamma_{f-2},\gamma_{f-1}, \gamma_{f+1}\}$ are
possibly equal to 2 and all other $\gamma_i\geqslant 3$ after
gaining support. Note that $\gamma_{f+2}\geqslant 3$ since
$N[u_{f+3}]\cap S\neq \es$. This yields
$\sum_{i=1}^n\gamma_i=5|S|\geqslant 3n-3$. For $n=5k$,
$|S|\geqslant\left\lceil\frac{3n-3}{5}\right\rceil=\left\lceil\frac{15k-3}{5}\right\rceil=\left\lceil3k-\frac{3}{5}\right\rceil=3k=\left\lceil
\frac{3n}{5} \right\rceil$, and $|S|\geqslant
\left\lceil3k+\frac{6}{5}\right\rceil=3k+2=\left\lceil \frac{3n}{5}
\right\rceil$ for $n=5k+3$.  This completes the proof.
\qed\end{proof}

\begin{definition}
\label{repeatedpatten} Let $S$ be a dominating set of $P_f(n,2)$. A
block $\mathcal{B}_i$ is called a self-contained block if
$\mathcal{B}_i\cap S=\{u_{i-2},v_i,v_{i+1}\}$ $($see
Figure~\ref{fig:selfcontained}$(\mathrm{a}))$.
\end{definition}

\begin{figure}[htb]
\begin{center}
\subfigure[A self-contained block]{
\includegraphics[scale=0.35]{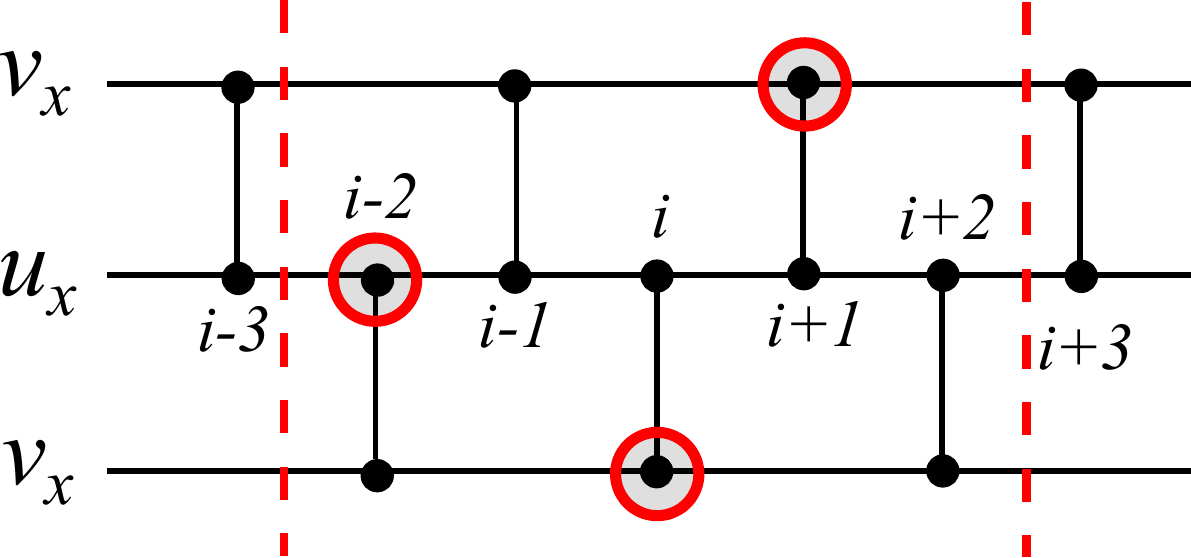}
} \quad \subfigure[Consecutive self-contained blocks]{
\includegraphics[scale=0.35]{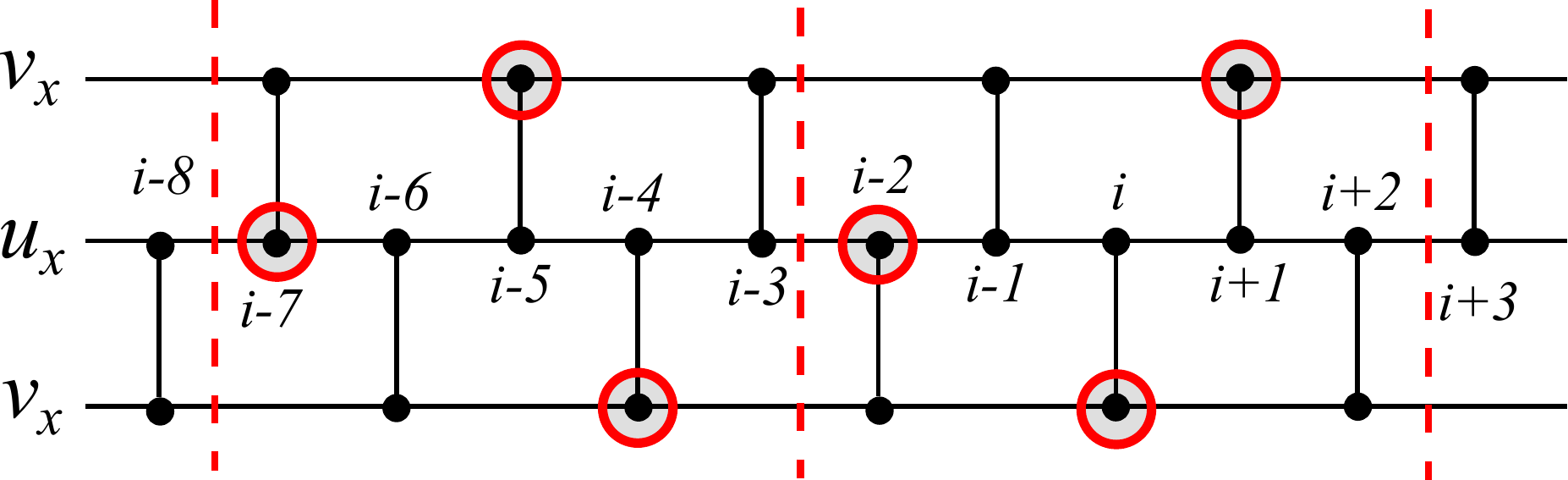}
}  \caption{\label{fig:selfcontained}Self-contained blocks.}
\end{center}
\end{figure}

\begin{proposition}\label{twoselfcontainedblocks}
If both $\mathcal{B}_i$ and $\mathcal{B}_{i-5}$ are self-contained
blocks, then $\gamma_x=3$ for $i-5\leq x\leq i$.
\end{proposition}
\begin{proof} By inspection (see Figure~\ref{fig:selfcontained}(b)), the proposition follows.
\qed\end{proof}

\begin{lemma}\label{lm:ri=3}
For $n=5k+4$ and $\gamma_f = 3$, if $S$ is a Type II set and any two
of $\gamma_{f-2}$, $\gamma_{f-1}$ and $\gamma_{f+1}$ are greater
than 2, then $|S|\geq\left\lceil \frac{3n}{5} \right\rceil$.
\end{lemma}
\begin{proof} Analogous to Lemma~\ref{5nn3}, only the elements in
$\{\gamma_{f-2},\gamma_{f-1}, \gamma_{f+1}\}$ are possibly equal to
2 and all other $\gamma_j\geqslant 3$ after gaining support. If any
two of $\gamma_{f-2}$, $\gamma_{f-1}$ and $\gamma_{f+1}$ are greater
than 2, then $\sum_{i=1}^n\gamma_i=5|S|\geqslant 3n-1$. By replacing
$n$ by $5k+4$, this yields $|S|\geqslant \left\lceil
3k+\frac{11}{5}\right\rceil=3k+3$. Clearly,
$\left\lceil\frac{3n}{5}\right\rceil=\left\lceil3k+\frac{12}{5}\right\rceil=3k+3$
when $n=5k+4$. Thus $|S|\geq\left\lceil \frac{3n}{5} \right\rceil$.
This completes the proof. \qed\end{proof}

\begin{lemma}\label{tm:ri=3}
For $n=5k+4$,  if $S$ is a Type II set, then $|S|\geq\left\lceil
\frac{3n}{5} \right\rceil$.
\end{lemma}
\begin{proof} Suppose to the contrary that $|S|< \left\lceil
\frac{3n}{5} \right\rceil$. We claim that exactly one of the
vertices in $\{v_{f-3},u_{f-3},v_{f-4},u_{f-4}\}$ is in $S$. We
argue the claim by contradiction and assume that there are two
vertices in $\{v_{f-3},u_{f-3},v_{f-4},u_{f-4}\}\cap S$. In this
case, if one of $v_{f-3}$ and $u_{f-3}$ and one of $v_{f-4}$ and
$u_{f-4}$ are in $S$, then $\gamma_{f-1}\geq 3$ and
$\gamma_{f-2}\geq 3$. By Lemma~\ref{lm:ri=3}, $|S|\geq\left\lceil
\frac{3n}{5} \right\rceil$, a contradiction. If
$\{v_{f-3},u_{f-3},v_{f-4},u_{f-4}\}\cap S=\{v_{f-4},u_{f-4}\}$,
then $\{v_{f-5},v_{f-3}\}\cap S\neq \es$ so that $v_{f-3}$ is
dominated. This results in $\gamma_{f-3},\gamma_{f-4}\geq 4$ and
$\gamma_{f-5}\geq 3$. Thus $\gamma_{f-1}$ can gain support from
$\gamma_{f-3}$, by Lemma~\ref{lm:ri=3}, $|S|\geq\left\lceil
\frac{3n}{5} \right\rceil$, a contradiction. Thus the claim holds
and $\gamma_{f-2}=2$. Accordingly, we have the following four cases
to consider.

\noindent {\bf Case 1.} $v_{f-3} \in S$.

In this case, $u_{f-5}, v_{f-6}\in S$ so that $u_{f-4}$ and
$v_{f-4}$ are dominated (see Figure~\ref{boundof5k4}(a)). This
results in $\gamma_{f-1} \geqslant 3$, $\gamma_{f-4} \geqslant 4$
and $\gamma_{f-6} \geqslant 3$ since $N[v_{f-7}]\cap S \neq \es$.
The minimum values of $\gamma_i$ for $i=f-6,f-5,\ldots,f+2$ are
$3,3,4,3,2,3,3,2$, and 3, respectively, and $\gamma_{f-2}$ can gain
support from $\gamma_{f-4}$. This leads to $\sum_{i=1}^n\gamma_i
=5|S|\geqslant3n-1$. By using a similar argument as in
Lemma~\ref{lm:ri=3}, $|S|\geqslant \left\lceil \frac{3n}{5}
\right\rceil$, a contradiction. Thus this case is impossible.

\noindent {\bf Case 2.}  $u_{f-3} \in S$.

Since $\gamma_{f-2}=2$ and $u_{f-3} \in S$, both $u_{f-4}, v_{f-4}
\notin S$ (see Figure~\ref{boundof5k4}(b)). Furthermore, $v_{f-6}
\in S$ so that $v_{f-4}$ is dominated. We claim that either
$u_{f-5}$ or $v_{f-5}$ is in $S$. If both $u_{f-5}$  and $v_{f-5}$
are not in $S$, then $\gamma_{f-3}=2$. However, $u_{f-3}$ is a
pseudo-couple vertex and, by Lemma~\ref{lmm:gamma=2}, $S$ does not
have the smallest couple number, a contradiction. Thus this claim
holds and $\gamma_{f-3}\geqslant 3$. Note that $\gamma_{f-6}
\geqslant 3$. The reason is that if $u_{f-6}\in S$, then
$\gamma_{f-6} \geqslant 3$; otherwise, by Lemma~\ref{lmm:P_i=3},
$\gamma_{f-6} \geqslant 3$. The minimum values of $\gamma_i$ for
$i=f-6,f-5,\ldots,f+2$ are $3,3,4,3,2,3,3,2$, and 3, respectively.
Thus $\gamma_{f-2}$ can gain support from $\gamma_{f-4}$. Thus this
case is also impossible.

\noindent {\bf Case 3.}  $v_{f-4} \in S$.

In this case, $v_{f-5} \in S$ so that $v_{f-3}$ is dominated (see
Figure~\ref{boundof5k4}(c)). We claim that $u_{f-5} \notin S$.
Suppose to the contrary that $u_{f-5} \in S$. This yields
$\gamma_{f-3}, \gamma_{f-4} \geqslant 4$ and $\gamma_{f-5},
\gamma_{f-6} \geqslant 3$. This results in the minimum values of
$\gamma_i$ for $i=f-6,f-5,\ldots,f+2$ to be $3,3,4,4,2,2,3,2$, and
3, respectively, and $\gamma_{f-2}$ and $\gamma_{f-1}$ can gain
support from $\gamma_{f-4}$ and $\gamma_{f-3}$, respectively. Hence
$|S|\geqslant \left\lceil \frac{3n-1}{5} \right\rceil=\left\lceil
\frac{3n}{5} \right\rceil$, a contradiction. Thus the claim holds.
When $u_{f-5} \notin S$, at least one vertex in $\{u_{f-6},
v_{f-6},u_{f-7},v_{f-7}\}$ must be in $S$ so that $u_{f-6}$ is
dominated. If two vertices in $\{u_{f-6}, v_{f-6},u_{f-7},v_{f-7}\}$
are in $S$, this results in $\gamma_{f-5}, \gamma_{f-6} \geqslant 4$
and $\gamma_{f-3}, \gamma_{f-4}, \gamma_{f-7}, \gamma_{f-8}
\geqslant 3$. This further implies that $|S|\geqslant \left\lceil
\frac{3n-1}{5} \right\rceil=\left\lceil \frac{3n}{5} \right\rceil$,
a contradiction. Thus at most one of $u_{f-6}, v_{f-6}$, and
$u_{f-7}$ can be in $S$. We claim that $v_{f-6}$ cannot be in $S$
either. If $v_{f-6}\in S$, then $u_{f-8}\in S$ to ensure that
$u_{f-7}$ is dominated. Moreover, $\gamma_{f-8} \geqslant 3$ no
matter whether $u_{f-8}$ is a pseudo-couple vertex or not. This
results in the minimum values of $\gamma_i$ for
$i=f-6,f-5,\ldots,f+2$ to be $4,3,4,3,2,2,3,2$, and 3, respectively.
Hence $|S|\geqslant \left\lceil \frac{3n-1}{5}
\right\rceil=\left\lceil \frac{3n}{5} \right\rceil$, a
contradiction. Thus the claim holds and only one of $u_{f-6}$ and
$u_{f-7}$ can be in $S$. If $u_{f-6}$ is in $S$, then, after
replacing $u_{f-6}$ by $u_{f-7}$, all vertices in $N[u_{f-6}]$ are
also dominated. Thus we only consider the case where $u_{f-7}\in S$.
Note that, in this case, $\mathcal{B}_{f-5}\cap S$ is exactly a
self-contained block. By repeating the above procedure on
$\mathcal{B}_{f-5x}$ for $2\leq x\leq k$, the only possible result
is that all $\mathcal{B}_{f-5x}$ are also self-contained blocks and
$\{v_{5k+2},v_{5k+3}\}\subset S$ (see Figure~\ref{boundof5k4}(d)).
By Proposition~\ref{twoselfcontainedblocks}, we have
$\gamma_{f-x}=3$ for $3\leqslant x\leqslant 5k-2$. Finally, we can
find that at least one vertex in $\{u_{5k+4},v_{5k+4},u_{5k+3}\}$
must be in $S$ so that $u_{5k+4}$ is dominated. This results in
$\gamma_x=4$ for $x=5k+2, 5k+3,5k+4$. Thus
$\gamma_{f-2},\gamma_{f-1}$ and $\gamma_{f+1}$ can gain support from
those vertices. This yields $\sum_{i=1}^n=5|S|\geqslant
3n=\left\lceil \frac{3n}{5} \right\rceil$, a contradiction.

\noindent {\bf Case 4.}  $u_{f-4} \in S$.

In this case, $v_{f-5}$ and $v_{f-6}$ are in $S$ so that $v_{f-3}$
and $v_{f-4}$ are dominated. Since $S'=S-u_{f-4}+v_{f-4}$ is still a
a Type II set with the smallest couple number which is already
considered in Case 3. Therefore, this case is also impossible. This
concludes the proof of the lemma. \qed\end{proof}

\begin{figure}[htb]
\begin{center}
\subfigure[Case 1]{
\includegraphics[scale=0.3]{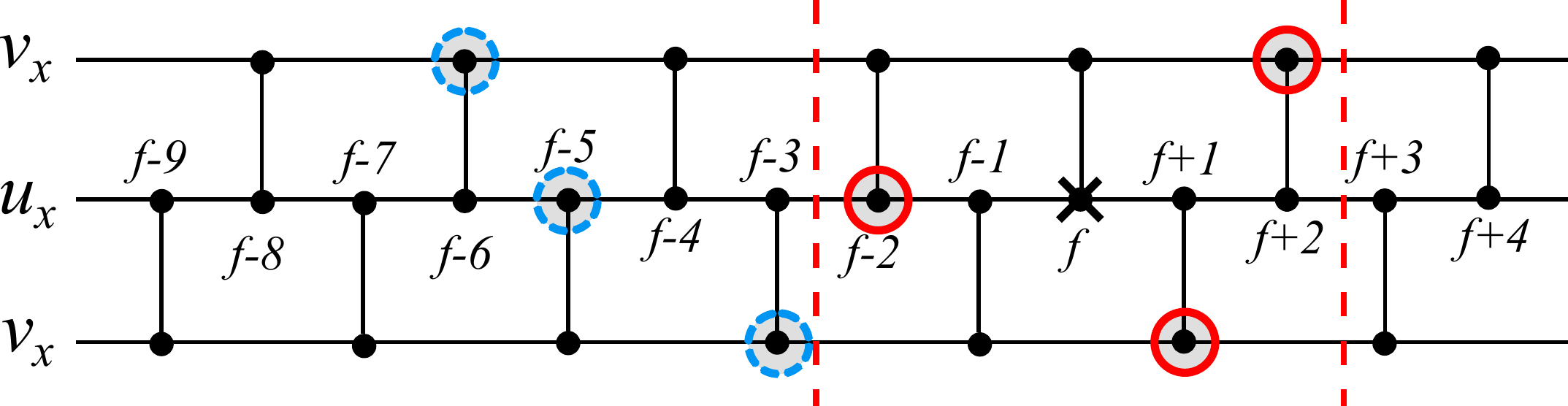}
} \quad \subfigure[Case 2]{
\includegraphics[scale=0.3]{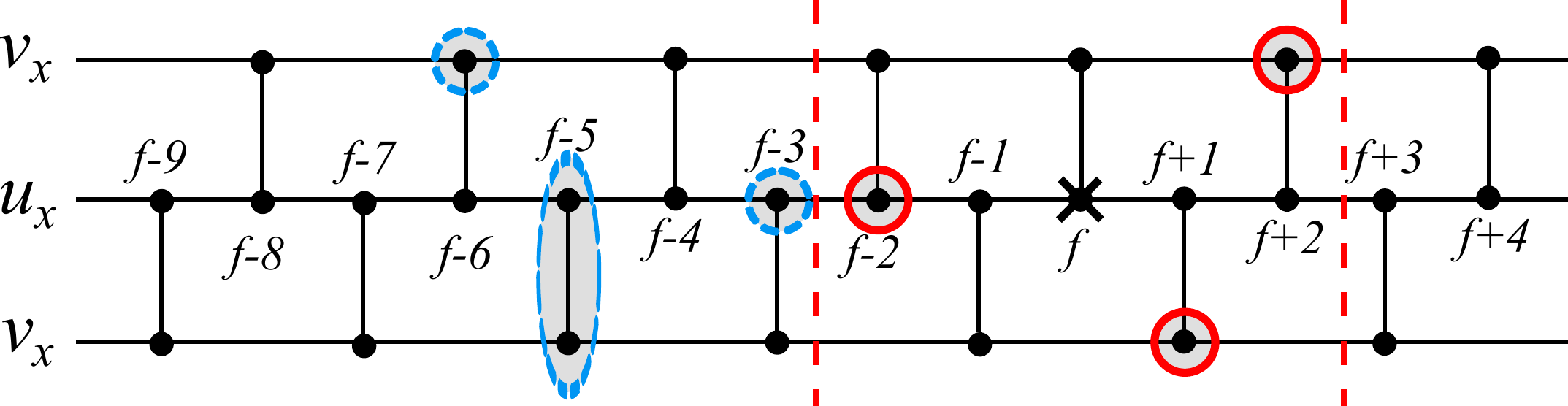}
} \quad \subfigure[Case 3]{
\includegraphics[scale=0.3]{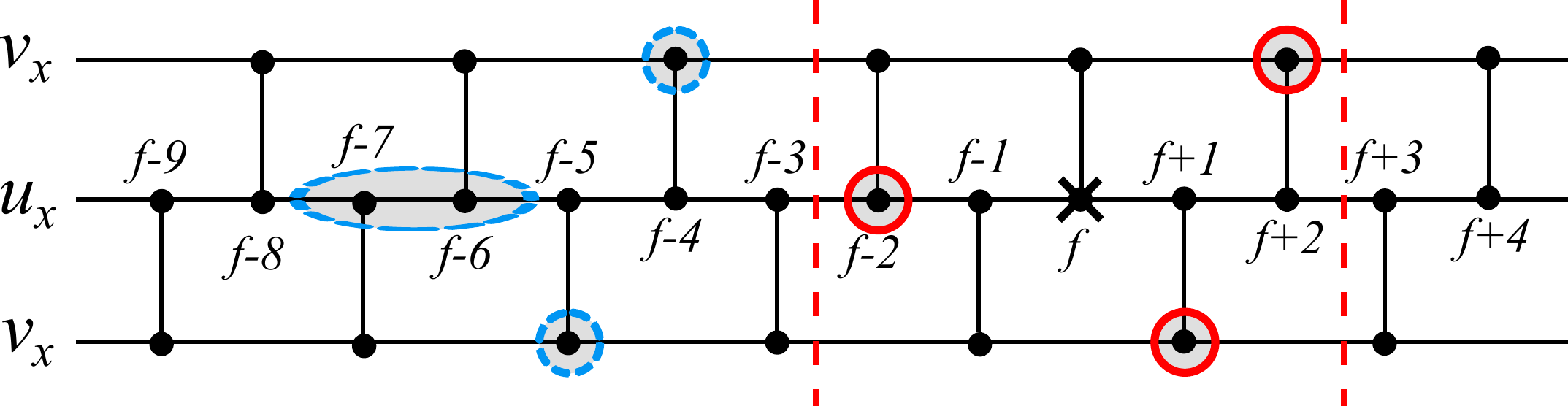}
} \quad \subfigure[The last four pairs of vertices]{
\includegraphics[scale=0.3]{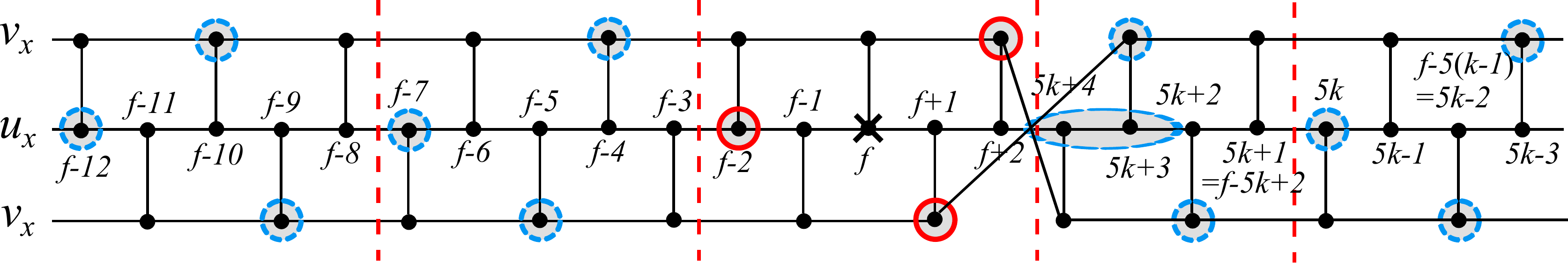}
} \caption{\label{boundof5k4} Illustrations for Lemma~\ref{tm:ri=3}.}
\end{center}
\end{figure}

\begin{lemma}\label{5kn+1+2}
For $n=5k+1$ and $n=5k+2$, if $S$ is a Type II set, then
$|S|\leq\left\lceil \frac{3n}{5} \right\rceil-1$.
\end{lemma}
\begin{proof} By using a
similar argument as in Case 3 of Lemma~\ref{tm:ri=3}, we can
construct a dominating set $S'$ of $P(n,2)$ when $n=5k+1$ (see
Figure \ref{fig-Kuo-6}(a)). Note that all $\gamma_i(S')=3$ for
$1\leqslant i\leqslant n$ except
$\gamma_{f-2}(S')=\gamma_{f-1}(S')=\gamma_{f+1}(S')=2$. Thus
$\sum_{i=1}^n\gamma_i(S')=5|S'|=3n-3=15k$ and $|S'|=3k$.  However,
$\left\lceil \frac{3n}{5} \right\rceil=3k+1$. This yields
$|S'|=\left\lceil \frac{3n}{5} \right\rceil-1$ and
$|S|\leq\left\lceil \frac{3n}{5} \right\rceil-1$.

Similarly, when $n=5k+2$, let $S''=S'\cup \{v_{5k+2}\}$ (see Figure
\ref{fig-Kuo-6}(b)). It can be verified easily that $S''$ is a
dominating set of $P(n,2)-u_f$. Note that all $\gamma_i(S'')=3$ for
$1\leqslant i\leqslant n$ except
$\gamma_{f-2}(S'')=\gamma_{f-1}(S'')=2$. Thus $5|S''|=3n-2=15k+4$
and $|S''|=3k+1$. However, $\gamma(P(n,2))=\left\lceil \frac{3n}{5}
\right\rceil=3k+2$. This yields $|S''|=\left\lceil \frac{3n}{5}
\right\rceil-1$ and $|S|\leq \left\lceil \frac{3n}{5}
\right\rceil-1$. This completes the proof. \qed\end{proof}

\begin{figure}[htb]
\begin{center}
\subfigure[$n=5k+1$]{
\includegraphics[scale=0.3]{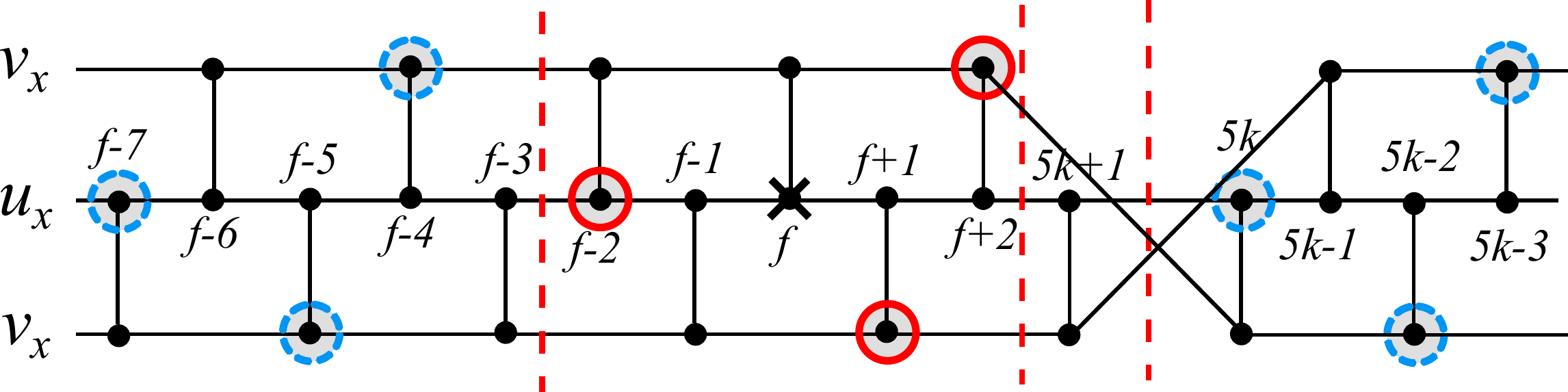}
} \quad \subfigure[$n=5k+2$]{
\includegraphics[scale=0.3]{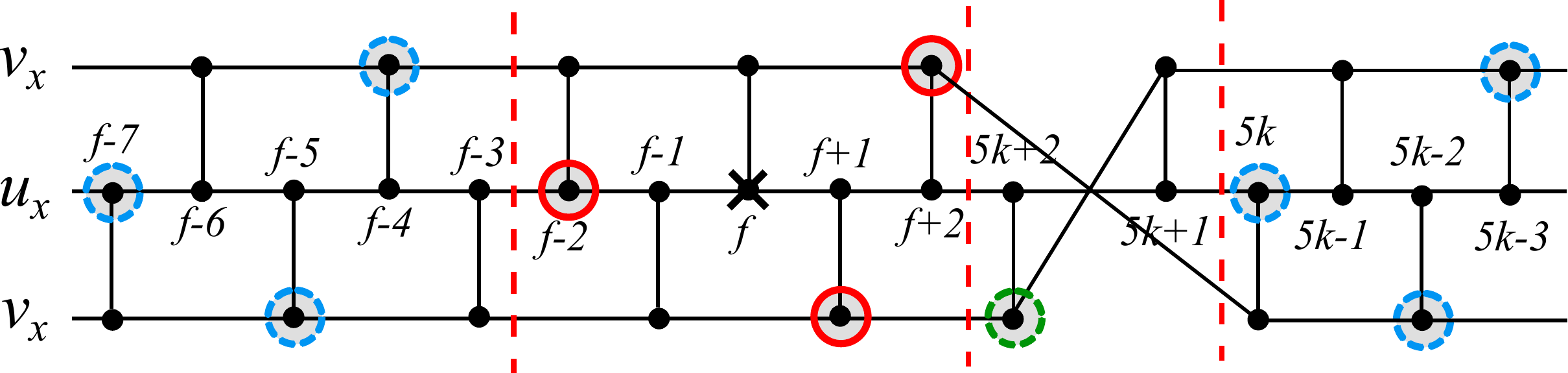}
}  \caption{\label{fig-Kuo-6}The dominating sets in $P_f(n, 2)$ when $n=5k+1$ and $5k+2$.}
\end{center}
\end{figure}

We summarize our results as the following theorem.

\begin{theorem}\label{mainresult}
Assume that $u_f$ is a faulty vertex in $P(n,2)$. Then for $n\geq 3$
\[\gamma(P_f(n,2))= \begin{cases}
\gamma(P(n,2))-1 & \quad \text{if $n=5k+1$ or $5k+2$} \\
\gamma(P(n,2)) & \quad \text{otherwise.}
\end{cases}
\]
\end{theorem}
\begin{proof}
By Corollaries~\ref{coro:gammaPn2} and \ref{lm:gamma_i>4} and
Lemmas~\ref{lm:TypeIIbcd} and \ref{5nn3},
$\gamma(P_f(n,2))=\gamma(P(n,2))$ when $n=5k$ and $5k+3$. By
Corollaries~\ref{coro:gammaPn2} and \ref{lm:gamma_i>4} and
Lemmas~\ref{lm:TypeIIbcd} and \ref{tm:ri=3},
$\gamma(P_f(n,2))=\gamma(P(n,2))$ when $5k+4$. By
Corollaries~\ref{coro:gammaPn2} and \ref{lm:gamma_i>4} and
Lemmas~\ref{lm:TypeIIbcd} and \ref{5kn+1+2}, $\left\lceil
\frac{3n}{5} \right\rceil-1$ when $n=5k+1$ or $5k+2$. This completes
the proof. \qed\end{proof}

\section{Concluding remarks}
\label{Conclusion}

In this paper, we show that $\gamma(P_f(n,2))= \gamma(P(n,2))-1$ if
$n=5k+1$ or $5k+2$; otherwise, $\gamma(P_f(n,2))=\gamma(P(n,2))$.
Our results can be applied to the alteration domination number of
$P(n,2)$. By Theorem~\ref{mainresult}, we can find the lower and
upper bounds for $\mu(P(n,2))$ as follows: $\mu(P(n,2))=1$ if
$n=5k+1$ or $5k+2$; otherwise, $\mu(P(n,2))\geq 2$. As a further
study, it is interesting to find out the exact value of
$\mu(P(n,2))$. On the bondage problem in $P(n,2)$, it is clear that
the domination number is still $\left\lceil \frac{3n}{5}
\right\rceil$ after removing any edge from $P(n,2)$. By
Theorem~\ref{mainresult}, we can find that $2\leqslant
b(P(n,2))\leqslant 3$ if $n=5k$, $5k+3$ or $5k+4$. It is also
interesting to find out the exact value of $b(P(n,2))$.


\end{document}